\documentclass{amsart}

\usepackage{amsmath,amsxtra,amssymb,amsfonts,latexsym, mathrsfs, dsfont,bm,mathtools}
\usepackage{latexsym}
\usepackage{bbm}
\usepackage[usenames]{color}
\usepackage[hyperindex]{hyperref}
\usepackage[initials]{amsrefs}
\usepackage{tikz}
\usepackage{comment}
\usepackage{color}
\usepackage{graphics,epsf}         
\usepackage{enumerate}

\newtheorem{theorem}{Theorem}[section]
\newtheorem{lemma}[theorem]{Lemma}

\newtheorem{definition}{Definition}[section]

\renewcommand{\q}{\quad}

\newcommand{\f}{\frac}

\newcommand{\p}{\partial}

\newcommand{\Id}{{\bf 1}}

\def\supp{{\text{\rm supp }}}

\newcommand{\bsb}{\boldsymbol}

\begin{document}

\title[The $L^p$ estimates]{Endpoint estimates for one-dimensional oscillator integral operators}

\date{\today}
\author{Lechao Xiao}
\address{ Department of Mathematics, University of Pennsylvania, Philadelphia, PA 19104, USA}
\email{xle@math.upenn.edu}
\begin{abstract}
The one-dimensional oscillatory integral operator associated to a real analytic phase $S$ is given by 
$$
T_\lambda f(x) =\int_{-\infty}^\infty e^{i\lambda S(x,y)} \chi(x,y) f(y) dy.
$$
In this paper, we obtain a complete characterization for the mapping properties of $T_\lambda $ on $L^p(\mathbb R)$ spaces, namely we 
prove that $\|T_\lambda\|_p \lesssim |\lambda|^{-\alpha}\|f\|_p$ for some $\alpha>0$ if and only if the point $(\f 1 {\alpha p} , \f 1 {\alpha p'})$ lies in the
reduced Newton polygon of $S$, and this estimate is sharp if and only if it lies on the reduced Newton diagram.  
\end{abstract}

\maketitle

\section{Introduction}\label{intro}
The very well-known and extremely useful tool in one-dimensional analysis is the following van der Corput lemma (see \cites{CCW99, ST93}):
\begin{lemma}\label{VD9}
For any real-valued function $u \in C^k(I) $ on some closed interval $I \subset \mathbb R$, if $|u^{(k)}(x)| \neq 0$ on $I$ $($assuming $u'$ is monotone when $k = 1$$)$, 
$$
\left|\int_I e^{i\lambda u(x)}dx\right | \leq C_k|\lambda|^{-\f 1 k}, \q \forall \lambda\,\,\in\mathbb R\,.
$$
\end{lemma}
This estimate shares the following two remarkable features: sharpness (the decay is best possible) and uniformity/stability   
(the constant $C_k$ depends only on the lower bound of $|u^{(k)}|$ in $I$ but not on other assumptions concerning $u$).  
In several variables (we mainly focus on two variables here) one has the following (see \cite{ST93}):
suppose $\chi\in C_0^{\infty}(\mathbb R^2)$ and $S$ is a real-valued function so that for some
$(k,l)\in \mathbb N^2$ not equal to $(0,0)$, 
\begin{align}\label{UI1}
\left|\f{\partial^{k+l}S}{\p_x^k \p_y^l}(x,y)\right| \neq 0 \q{\rm for\,\, all } \q (x,y)\in\supp \chi. 
\end{align} 
Then 
\begin{align}\label{VD01}
\left |\iint_{\mathbb R^2} e^{i\lambda S(x,y)}\chi(x,y) \, dxdy \right | \leq C(S) \cdot |\lambda|^{-\frac 1 {k+l}}.  
\end{align}
However, this two dimensional analogue is less satisfying for the estimate is neither sharp in general 
 (consider $S(x,y)= x^k y^l$, for example) nor uniform ($C(S)$ depends on higher derivatives of the phase).   
 There has been significant interest in the harmonic analysis literature to develop a robust
and general theory of high-dimensional oscillatory integrals that shares the above two features.  
However, progress on this problem has been slow, 
because, among many other reasons, the singularities of the phases involved may themselves be substantially more complicated.
Much progress has been made in the feature of uniformity/stability. 
Results in this category include \cites{CCW99, CW02, PSS01, PSS99, CLTT05, GR8, GR11,GRB11, KA86}. 
The goal of the present paper is trying to understand the other feature, namely the feature of sharpness. More precisely, 
we are interested in an oscillatory integral model that is intrinsically associated to \eqref{UI1}, 
in the sense its sharp decay rate estimates are almost equivalent to the assumption \eqref{UI1}, 
but restrict ourself to analytic phases.

Through out the rest of this paper, $\chi$ denotes a function belonging to $C_0^{\infty}(\mathbb R^2)$ and the phase $S$ is 
 real analytic in $\supp \chi$, in the sense $S$ is locally equal to its Taylor expansion in $\supp \chi$.  
The subjects under consideration are the $1+1$ dimensions oscillatory integral operator 
\begin{align}\label{mo}
T_\lambda f(x) =\int_{-\infty}^{\infty} e^{i\lambda S(x,y)} \chi(x,y) f(y) dy.
\end{align}
We are interested in decay rates (as $\lambda\to \pm \infty$ ) for the norm of $T_\lambda$ as an operator that maps $L^p(\mathbb R)$ into itself for all $p\geq 1$.  
For convenience, we use $\|T_\lambda\|_p$ to denote this norm and use $p'$ to denote the conjugate exponent of $p$, 
that is $\f 1 {p'} =1 -\f 1 p$.  
The $k=0$ (or $l=0$) cases are uninteresting, for $S(x,y)$ can be 
a function of $y$ alone ($x$ alone res.) and
decay estimates for $\|T_\lambda\|_p$ may not exist for any $p$.  
Henceforth, through out the rest of this paper, we assume $k$ and $l$ are both in $\mathbb N_* = \mathbb N\backslash \{0\}$.  
Our first result states that the assumption (\ref{UI1}) is indeed sufficient to 
obtain a sharp $L^p$-estimate for $T_\lambda f$.    
\begin{theorem}\label{M1}
If \eqref{UI1} holds for some $(k, l)\in\mathbb N_*^2$,
then there is a constant $C$ independent of $\lambda$ such that 
\begin{align}\label{KL0}
\|T_\lambda\|_{\frac{k+l}{k}}  \leq C |\lambda|^{-\f 1 {k+l}}. 
\end{align}
\end{theorem}
Given the natural essence of this estimate and the fact its scalar analogue \eqref{VD01} is well-known,     
it is somewhat surprising that (\ref{KL0}) hasn't been obtained before. 
Indeed, \eqref{KL0} is significantly deeper than (\ref{UI1})
and its proof is highly non-trivial, involving various ideas and techniques from many previous works. 
The sharpest results in light of \eqref{KL0} to date, to the author's best knowledge, are due to Phong, Stein and Sturm \cite{PSS01}, 
where they obtained almost sharp estimates (with a power of $\log |\lambda| $ loss) for polynomial phases; see the discussion below.   
In addition, 
the estimate \eqref{KL0} also strictly improves upon Phong and Stein's seminal work \cite{PS97}, concerning sharp $L^2$-estimates, in the sense
\eqref{KL0} is the endpoint/extreme point estimates that can be used to 
interpolating their results; see Theorem \ref{COM}. While Varchenko's estimate \cite{VAR76}, 
namely the scalar analogue of Phong and Stein, can not be obtained from (\ref{UI1}) due to the fact that (\ref{UI1}) is not sharp in general.   
Moreover, the results in Theorem \ref{M1} are complete, in the sense they 
provides all possible estimates for $\|T_\lambda\|_p$ that one can obtain basing on assumptions like (\ref{UI1}).
To illustrate this point, notice first there are two trivial estimates $\|T_\lambda\|_\infty \leq C$ and $\|T_\lambda\|_1 \leq C$. 
Interpolating among (\ref{KL0}) and them yields estimates (\ref{KL0})
for all pairs of real numbers $(k',l')\geq^* (k,l)$. Here the partial order $(A,B)\geq^* (a,b)$ indicates $A\geq a$ and $B\geq b$ simultaneously.  If we set 
$\mathcal A =\{(k,l)\in\mathbb N_*^2\,\, :\,\, (\ref{UI1}) \}$, then by interpolation, the estimates (\ref{KL0}) hold for all pairs of real numbers $(k',l')$ belonging to the convex hull
of the union of all positive quadrants attached to each element in $\mathcal A$. 
This convex hull is essentially the largest set one can establish estimates like \eqref{KL0}. That is to say the ``inverse'' of Theorem \ref{M1} is almost true. 
However, it is a subtle issue 
when certain mix-derivatives of $S$
vanish at some point(s) on the boundary of $\supp \chi$. To avoid this subtlety,  
it is more appropriate to phrase our results locally.  

For each $(x_0,y_0)\in \supp \chi$, $S(x,y)$ is locally equal to its Taylor expansion, i.e.  
there is a neighborhood of $(x_0,y_0)$ on which 
\begin{align}\label{phase01}
S(x,y) = \sum_{p,\, q\geq 0}c_{p,q}(x-x_0)^p(y-y_0)^q.
\end{align}
Let $\phi\in C_0^{\infty}(\mathbb R^2)$ be supported in a sufficiently small neighborhood $U$ of $(x_0,y_0)$
which is non-vanishing at $(x_0,y_0)$, 
 and let $T_{0}$ denote the corresponding localization of $T_\lambda$, i.e. the operator given in (\ref{mo}) with the cut-off $\chi$ replaced by $\phi$. 
We use $\mathcal N^*(S)$ to denote the {\it reduced} Newton polygon associated to $S$ at $(x_0,y_0)$, i.e. the convex hull of the union of all quadrants $[p,\infty)\times [q,\infty)$ with $c_{p,q}\neq 0$ and $(p,q)\in\mathbb N_*^2$. 
The pairs $(p,0)$ and $(0,q)$ are irrelevant to any  $\|T_0\|_p$ as mentioned above.  
The boundary of $\mathcal N^*(S)$, denoted by $\mathcal D^*(S)$, is called the reduced Newton diagram of $S$ at $(x_0,y_0)$. 
The natural connection between the operator $T_0$ and the 
Newton polygon of the phase is captured by the following theorem. 
\begin{theorem}\label{COM}
Assume $\alpha>0$. 
Under the above setting, 
$
\|T_0\|_p \leq  C|\lambda|^{- \alpha}
$
 iff $(\f 1 {p\alpha}, \f 1 {p'\alpha})\in \mathcal N^*(S)$ 
 and this estimate is sharp iff $(\f 1 {p\alpha}, \f 1 {p'\alpha})\in \mathcal D^*(S)$.
\end{theorem}
The content of this theorem can be depicted by the graph below. The reduced Newton diagram (the thickened line segments) 
is the set of critical points that one has sharp decay estimates. 
 Residing on its right (the gray area) is the set of points $(\f 1 {p\alpha}, \f 1 {p'\alpha})$
for non-sharp estimates of  $\|T_0\|_p \leq  C|\lambda|^{-\f 1 \alpha}$. 
On its left (the white area) 
is the set of points $(\f 1 {p\alpha}, \f 1 {p'\alpha})$ that the estimate $\|T_0\|_p \leq  C|\lambda|^{-\f 1 \alpha}$ cannot hold.  
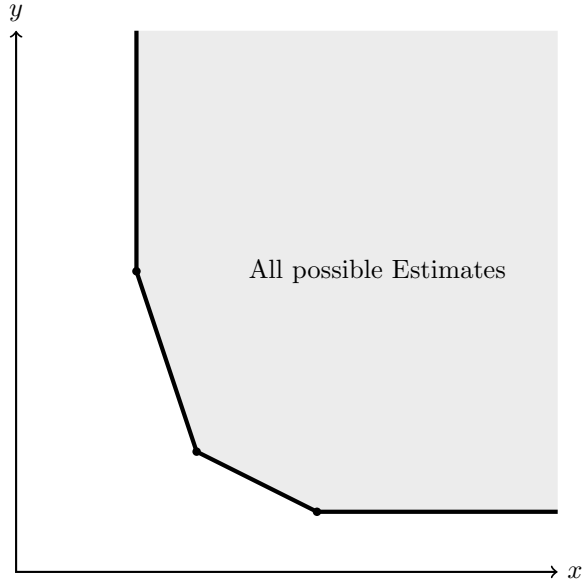
\begin{center}
\begin{tikzpicture}[scale=0.8]
\draw [<->,thick] (0,9) node (yaxis) [above] {$y$}
        |- (9,0) node (xaxis) [right] {$x$};
     \fill[gray!15] (2,9)--(2,5)--(3,2)--(5,1)--(9,1)--(9,9);
       \fill (5,1)   circle (2pt) (3,2) circle (2pt) (2,5)  circle (2pt);
        \draw [ultra thick]	(9,1)--
        			(5,1) --
			(3,2) --
			(2,5) -- 
			(2,9);
	\node at (6,5) {All possible Estimates};

\end{tikzpicture}
\end{center}

\vspace{0.2in}

For convenience, we refer Theorem \ref{M1} as the global estimates and Theorem \ref{COM} as the local estimates.  
These two estimates are essentially equivalent.   
It is quite obviously
that Theorem \ref{M1} implies the sufficient part of the first conclusion of Theorem \ref{COM}, given the support of $\phi$ is sufficiently small. 
On the other hand, 
the global estimates can also be obtained from the local estimates via an argument of smooth partition of unity.  
Indeed, for each $(x_0,y_0)\in\supp \chi$ and 
each smooth cut-off $\phi$ supported in a sufficiently small neighborhood of it, 
Theorem \ref{COM} yields $\|T_0\|_p \leq C|\lambda|^{-\f 1 \alpha}$, provided the pair $(\f 1 {p\alpha},\f 1 {p' \alpha})$ 
lying in the reduced Newton polygon of $S$ at $(x_0,y_0)$. 
The assumption \eqref{UI1} guarantees that $(k,l)$ is one such pair 
and that \eqref{KL0} holds for each $T_0$ associated to each $(x_0, y_0)$.
 Theorem \ref{M1} is then obvious 
for $\chi$ is compactly supported. 
Therefore, it suffices for us to establish Theorem \ref{COM}, which will be accomplished via interpolation. 
Once sharp estimates were established for all endpoints, namely for all $(p,\alpha)$ such that 
$(\f 1 {p\alpha}, \f 1 {p'\alpha})$ is a vertex of $\mathcal N(S)$, all other estimates in Theorem \ref{COM}
will follow from interpolating among these endpoints and the two trivial estimates mentioned above.


Oscillatory integrals of the form (\ref{mo}) and related objects 
have attracted considerable interest during the last half century. 
Besides its intrinsic interest, the decay rate of $\|T_\lambda\|_p$ is closely related to the regularity of Radon transforms;
see \cite{PS94-2,SEE99, Y04} and the references there. 
When $(k,l) =(1,1)$, 
the phase is non-degenerate
and (\ref{KL0}) corresponds to the classical results of H\"ordmander \cite{HOR73}. 
The general degenerate cases are substantially more complicated. 
Sharp estimates for the scalar integral \eqref{VD01} were established by Varchenko \cite{VAR76} in 1976 for arbitrary analytic phases, given $\chi$ has sufficiently small support. 
He showed that the decay is characterized in terms of the Newton polygon of the phase\footnote{under an appropriate coordinate system, known as the adapted coordinates}, 
namely the so-called Newton distance, 
confirming an early hypothesis of V. I. Arnold. 
From this view of point, (\ref{UI1}) alone is not sufficient to capture the sharp decay of (\ref{VD01}). 
In another word, (\ref{VD01}) is not a perfect model for analogues of the van der Corput lemma in high dimensions.  
The operator (\ref{mo}) seems to be more suitable, for the associated estimate (\ref{KL0}) captures the assumption (\ref{UI1})
naturally. 
A systematic study for 
(\ref{mo}) for general degenerate phases was pioneered by Phong and Stein in the `90s.
 The key analytic tool to unlock the mystery is, 
 now commonly referred as, Phong-Stein's operator van der Corput \cite{PS94-1,PS94-2}. Sharp $L^2$ estimates were first established for homogeneous polynomial phases \cite{PS94-2}, extended to arbitrary real analytic phases \cite{PS97} and finally to the damped version \cite{PS97-2}. 
The first two results (without damping) can be interpreted as a particular case of Theorem \ref{COM}, namely 
the case when $p=2$ and $(\f 1 {2\alpha}, \f 1{2\alpha})$ 
is lying on the reduced Newton diagram.  
This should not be surprising for 
our proof employs various ideas from these papers and many others. 
One of the advantages of Theorem \ref{M1}, compared to Varchenko's or Phong-Stein's results, is that we do not require the full knowledge of the Newton polygon
(namely the Newton distance) 
to describe a sharp estimate. One single vertex, i.e. (\ref{UI1}) is sufficient for this purpose. 

In the category of $C^\infty$, 
almost sharp $L^2$ estimates was implicitly contained in Seeger \cites{SEE93, SEE99}, whose arguments are quite different from Phong and Stein's.  
Combining the ideas from Phong-Stein and Seeger, Rychkov \cite{RY01} was able to obtain sharp estimates in most $C^\infty$ situations. The exceptional cases are the ones when the phases are completely degenerated, which is also one of the two major difficult cases in our proof of Theorem \ref{COM}.   
The full $C^\infty$ situation is indeed quite challenging, for 
the formal power series of the phase alone may not be sufficient to capture the behavior of the singularities.  
 In \cite{GR05}, Greenblatt,  
 via a remarkable stopping time argument,
 was able to establish the desired sharp $L^2$ estimates in full generality. 
Whether Theorem \ref{COM} (and thus Theorem \ref{M1}) can be extended to full $C^\infty$ situations are still under investigation.

In a slightly different context, Phong, Stein and Sturm \cite{PSS01} 
studied uniform estimates for certain multilinear oscillatory integral operators associated to polynomials. 
The global estimates indeed correspond to their {\it bilinear} (i.e. $d=2$) setting below.  
They defined the following multilinear form 
\begin{align}\label{PSS9}
T_D(f_1,\dots,f_d) = \int_D e^{i\lambda S(x_1,\cdots, x_d)} \prod_{j=1}^d f_j(x_j)d x_1\cdots d x_d
\end{align}
and proved in particular 
\begin{align}\label{PSS8}
|T_D(f_1,\cdots,f_d)| \leq C |\log (2+|\lambda|)|^{d-\f 1 2}
|\lambda|^{-\f 1 {|\alpha|}} \prod_{j=1}^d\|f_j\|_{p_j}
\end{align}
under the assumptions that $S: [0,1]^d\to \mathbb R$ is a polynomial of degree $n$, 
that $\alpha =(\alpha_1,\cdots, \alpha_d)\in \mathbb N^d$ is a multiindex,
that $D$ is the subset of $[0,1]^d$ defined by $|S^{(\alpha)}| \geq 1$ and that $\f 1 {p_j} = 1 -\f {\alpha_j}{|\alpha|}$.  
Their estimates are uniform in 
a sense that the constant $C$ in \eqref{PSS8} depends only on $d$, $n$ and $\alpha$, but may not be sharp due to the $\log(2+|\lambda|)$ term.      
While in the {\it bilinear} setting when $p_1 $ and $p_2 $ are all equal to 2, by inserting a smooth cut-off, they were able to eliminate the $\log$ term and thus obtained sharp and uniform estimates at the same time. 
Unfortunately, their treatment seems not sufficient to eliminate the $\log$ term for all the endpoints, even when the phase is a monomial. 
We will explain this point in Section \ref{SOP}. 
Moreover, their approach makes essential use of the polynomial character of the phases, namely Bezout's Theorem and
certain uniform estimates of polynomials (see Lemma 1.2 \cite{PS94-2}), and does not seem to apply to the more general setting.  
However, there has been evidence that the endpoint estimates may be true.   
Under the assumption that $S$ in (\ref{phase01}) is a homogeneous polynomial 
of degree $n$ 
in $(x-x_0)$, $(y-y_0)$
with $c_{1,n-1}c_{n-1,1} \neq 0$, 
Yang \cite{Y04} was able to obtain sharp endpoint estimates for $T_0$; see \cite{GS99} as well.  
Very recently, Shi and Yan \cite{SY14} established sharp endpoint estimates for arbitrary homogeneous polynomial phases. 
In the real analytic category, Yang \cite{Y05} showed that if $\|T_0f\|_p\leq C|\lambda|^{-\f 1 \alpha}\|f\|_p $ then $(\f 1 {p\alpha}, \f 1 {p' \alpha})$ must lie in $\mathcal N(S)$ and established all such estimates except the endpoints. 
Theorem \ref{COM} states that all endpoint estimates are indeed true.   

Of course, the ultimate goal is to establish sharp and uniform estimates of \eqref{KL0}. However
our proof of Theorem \ref{COM} (and thus of Theorem \ref{M1}) relies crucially on the resolution algorithm developed in \cite{X2013},
which itself seems not sufficient to obtain optimal and uniform estimates simultaneously. In particular, the following question remains 
wide-open: 

\textbf{Problem:} Suppose $S$ is smooth such that $| \frac  {\p ^{k+l}}{\p_x^k\p_y^l}S| \geq 1$ in $[-1,1]^2$. Suppose also $\supp\chi$ is contained in $[-1,1]^2$. Does (\ref{KL0}) hold for some constant $C$ independent of other assumptions on $S$?    

It may be too ambitious to address this problem in such generality. 
A more realistic goal, but still quite challenging, is to prove such results under additional assumptions that $S$ is real analytic and that $C$ is allowed to depend on higher derivatives of $S$. 
We refer the readers to \cites{KA86, PSS99} for discussion concerning stable and sharp estimates for the scalar integral \eqref{VD01} and its sublevel set analogue; see \cite{PS00} as well. 
Another nice investigation of this problem can be found in Carbery, Christ and Wright \cite{CCW99}, where they obtained the uniform estimate $\|T_\lambda\|_2\leq C|\lambda|^{-\f1 {2l}}$
 when $k=1$ and $l\geq 2$.  
However, the $|\lambda|^{-\f1 {2l}}$ decay still falls short of the $|\lambda|^{-\f 1 {l+1}}$ decay rate
suggested by the optimal decay in Theorem \ref{M1}. 
In sum, there is still a huge gap between the expected optimal and uniform estimates and what is currently known.  
This line of investigation may deserve further effort. 

The rest of this paper is organized as follow. 
In the next section, we will carefully describe our strategy in proving Theorem \ref{COM}. 
We will review Phong-Stein(-Sturm)'s arguments and then explain by examples why their arguments are not sufficient 
to build up Theorem \ref{COM} in full generality. 
Section \ref{pre01} provides the analytic tools for our main course of the proof. 
In section \ref{res}, we will sketch the proof of the resolution algorithm but refer the readers to \cite{X2013}, as well as \cite{GX15} for rigorous details. 
The significance of the algorithm is twofold. 
First, it reduces our problem to the situation that the phase is essentially a monomial; see Section \ref{Red}.  
Most of the cases can be then addressed by carefully exploiting the orthogonality of the operator, 
whose details will appear in Section \ref{minor}.  
Second but most importantly, it helps to identify precisely when and where exceptional cases arise and 
deduce crucial algebraic structure for the phases in such cases. 
With this in place, the first exceptional case is addressed by a lifting trick (Lemma \ref{INT001}), which 
reduces the desired $L^p$-estimate 
to a $L^2$-estimate for a truncated non-degenerate oscillatory integral operator. 
The other exceptional case is handled 
by complex interpolation between
a $L^2$ estimate for damped oscillatory integral operator 
and a variant of $H^1$ estimate for some oscillatory singular integral operator.  
Full details will appear in Section \ref{major1} and \ref{major2} respectively.

\subsection*{Acknowledgements} 
The author would like to thank Xiaochun Li and Philip T. Gressman for many helpful suggestions in this paper, and Zuoshunhua Shi for early collaboration on this topic.

\section{Strategy of the proof}\label{SOP}
To understand our strategy, 
we will begin with reviewing Phong-Stein(-Sturm)'s arguments for handling $T_\lambda $. In particular we will illustrate, by model examples,  
why their arguments fail in two exceptional cases.  
Then we will outline the methods to settle them.  
Of course the key is to deduce useful algebraic structure for the phases when exceptional cases occur,
but this will only be clear after running both the resolution algorithm and their arguments.  

For simplicity, we begin with $p = p'=2$.  
The three basic principles are
the size-estimate, the oscillation-estimate and the almost-orthogonality:
\begin{enumerate}
\item[(A.1)] If $\chi$ is supported in a $\sigma\times\rho$ box, then $|\langle T_\lambda f, g \rangle|\lesssim \sigma^{\f1{ 2}}\rho^{\f 1 2} \|f\|_2\|g\|_{2}$;
\item[(A.2)] If $|S''_{xy}|\sim 2^{-j}$ in $\supp \chi$, then $|\langle T_\lambda f, g \rangle|\lesssim |2^{-j}\lambda|^{-1/2} \|f\|_2\|g\|_{2}$;
\item[(A.3)] If we have a sequence of operators $T_h$ defined as in \eqref{mo} with $\chi_h$ as the smooth cut-offs, and if 
both the $x$-projections and $y$-projections of $\{\supp \chi_h\}$ are pair-wisely disjointed, then 
 $ \sum_h |\langle T_h f, g\rangle | \lesssim \sup_h |\langle T_h f, g\rangle |$.
\end{enumerate} 
Rigorous statements can be found in Lemma \ref{size1}, Lemma \ref{os2} and Lemma \ref{or3}, respectively. 
The size-estimate (A.1) is used in low-oscillation regions, i.e. when $j$ is large, while (A.2) is employed in high-oscillation regions.  
The almost-orthogonality (A.3) is employed in the summation 
of the operators in the regions of same oscillation.  
The two main procedures in their arguments are decomposition and summation. 
In the first procedure, one decomposes $\chi$ into the sum of collections of cut-offs adapted to the three principles above  
\begin{enumerate}
\item each cut-off is essentially supported in a rectangular box,
\item for each $j\geq 0$, $|S_{xy}''|\sim 2^{-j}$ in the support of the cut-offs belonging to the $j$-th collection,
\item cut-offs in the same collection 
 exhibit the orthogonality properties as (A.3).
\end{enumerate}
   The operator $T_\lambda $ is decomposed into small pieces accordingly.  
This decomposition procedure is the most crucial part to bound the operator $T_\lambda $ and how to obtain such decomposition is a very delicate issue.  
   
 In the second procedure, we estimate $T_\lambda $ piece by piece using either (A.1) or 
 (A.2) (depending on which one is superior) and carefully sum over all of them.  
The almost-orthogonality (A.3) asserts that the bound obtained by estimating one single piece
is essential the same as the one obtained by 
 adding all pieces together in the same collection. 
Thus the summation (and decomposition) inside the same collection is not essential but  
cross-collection summation (and decomposition) is the main issue.  
To continue, there is a threshold value $j_0$ such that for each $j<j_0$ (i.e. in the high-oscillation regions) and for each piece in the $j$-th collection,  the oscillation-estimate (A.2) is superior and has a bound $\sim 2^{(j-j_0)/2}|2^{-j_0}\lambda|^{-1/2}$. 
The crucial extra decay term $2^{(j-j_0)/2}$ paces the way for the summation over $j< j_0$. 
As it will be observed later, 
there is some $p\neq 2$ such that there won't be such extra decay in any high-oscillation regions.   
In the $j>j_0$ case, the size-estimates (A.1) are superior with bounds $\sim 2^{\epsilon(j_0-j)}  |2^{-j_0}\lambda|^{-1/2}$ for some $\epsilon>0$.   
We can then obtain the sharp bound $ |2^{-j_0}\lambda|^{-1/2}$ by summing over all $j$ in these two regions.  
The arguments of Phong-Stein-Sturm \cite{PSS01} of handling polynomial phases are slightly different. 
The first step is to decompose the support of $\chi$ into two regions. 
In the $|S''_{xy}|\leq 2^{-j_0}$ region, they were able to control the bound of the operator by certain sublevel set estimates; see Lemma 2.1\cites{PSS01} and Lemma 3.8 \cite{CCW99}. The $|S''_{xy}|\geq  2^{-j_0}$ region was further decomposed into subregions $|S''_{xy}|\sim  2^{-j}$, $j\leq j_0$. They developed a more flexible version of the oscillation-estimate (A.2), obtaining the same bounds $\sim 2^{(j-j_0)/2}|2^{-j_0}\lambda|^{-1/2}$ for the operator 
without the need of sub-dividing $|S''_{xy}|\sim  2^{-j}$ into boxes.  However,  
the latter arguments make essentially use of the polynomial feature of the phases; see Lemma 1.2 \cite{PS94-2}.  

As long as one is able to construct a nice decomposition of the operator $T_\lambda $, 
one can generalize the above arguments to the $p\neq 2$ case. Under the same assumptions in (A.1), (A.2) and (A.3) we have the following analogues 
\begin{enumerate}
\item[(B.1)] $|\langle T_\lambda f, g \rangle|\lesssim \sigma^{\f1{ q}}\rho^{\f 1 {q'}} \|f\|_q\|g\|_{q'}$ for 
$1\leq q \leq\infty$;
\item[(B.2)] The estimates obtained by interpolating (A.2) with the $q=1$ estimates in (B.1) when $p<2$ and with the $q=\infty $ estimates when $p>2$; 
\item[(B.3)] Same as (A.3),
\end{enumerate} 
respectively. 
To get into the heart of our problem, we shall move one step further into the decomposition process, 
which is often referred as (an algorithm for) resolution of singularities; see Hironaka \cite{HI64} and many others
 \cites{PS97, PSS01, PSS99, GR04, X2013, IM11, CGP13, VAR76 }. 
Our proof of Theorem \ref{COM} is built up upon the one developed in \cite{X2013}.  
The goal of this algorithm is to decompose a small neighborhood of a singularity into finitely many subregions in each of 
which the function $S''_{xy}$ behaves like a monomial. 
The problem is reduced to bound the operator in every single subregion 
on which $S''_{xy}$ is essentially a monomial.  
We divide each subregion into rectangular boxes via a bi-dyadic decomposition in the variables.
Now in each of such box, $S''_{xy}$ is comparable to a fixed value.  
One can then apply the above mentioned strategy coupled with (B.1-B.3) to control the operator.  
It turns out that this approach works quite effectively in most situations, which we refer as the minor case. 
While in the remaining situations, although we still have good control for low-oscillation regions (using (B.1)), these arguments break down completely in the sense,
there won't be extra decay (compared to a $2^{(j-j_0)/2}$ decay for $p=2$) in any high-oscillation regions. 
Consequently, one has to give up the attempt of dividing $\supp\chi$ into regions with $S''_{xy}\sim 2^{-j}$. 
We refer such situation as the major cases and there are two of them. 
The significance of the algorithm is that, after reexamining the parameters appearing in the iterations, 
we deduce precisely when and where these two major cases will occur. Consequently, we are able to extract structural information of the phase. 
The first major case occurs when 
$k=1$ (or $l=1$ by adjointness) near the $x$-axis ($y$-axis rep.).
In another word, it happens when $S''_{xy}$ is singular in the ``highest order'' near the $x$-axis ($y$-axis rep.). 
 This case is produced in the initial stage of iterations of the algorithm and 
a model case is given by $S''_{xy}(x,y) =y^{l-1}$.      
The second case also occurs when $k=1$ (or $l=1$) near some analytic curve $y=\gamma(x)$ with $S''_{xy}(x,y)= (y-\gamma(x))^{l-1}$. In another word, it happens when $S''_{xy}$ is singular in the ``highest order'' near $y =\gamma(x)$.
This case is produced in higher stage of iterations and a model case is given by $S''_{xy}(x,y) =(y-x)^{l-1}$.
 One shall also compare this case with the completely degenerate case in \cite{RY01}.     
To see why Phong-Stein(-Sturm)'s approach does not work in the first model case, we divide the support of $\chi$ into regions with $S''_{xy} =y^{l-1} \sim 2^{-j}$, i.e. $y\sim 2^{-j/(l-1)}$.  Then $|\langle T_\lambda f, g\rangle |$ is controlled above by $|2^{-j}\lambda|^{-1/2 }\|f\|_2\|g\|_2$ and 
$|2^{-j/(l-1)} \|f\|_\infty \|g\|_1$ as well. 
By interpolation, it is then majorized by $|\lambda|^{-\f 1 l } \|f\|_{l+1} \|g\|_{\f {l+1}{l}}$, which contains no decay in $j$. 
Similar phenomenon occurs in the second model case.  
Henceforth, one must seek for alternative methods to handle these two major cases.  
\begin{enumerate}
\item  
A key idea to settle the first major case comes from Shi and Yan\cite{SY14}. They observed that, when $S''_{xy}(x,y) =y^{l-1}$, via a lifting trick from Zygmund
\cite{ZY}, 
the $L^{l+1}$ estimates can be deduced from the $L^2$ boundedness of the Fourier transform, 
which is obvious due to Plancherel's theorem. 
In the general real analytic setting, the problem is reduced to the $L^2$ estimates for a truncated non-degenerate oscillatory integral operator. This truncation prevents us from applying the results or arguments from \cite{HOR73}. 
To address this issue, we rely on a clever idea from Phong Stein and Sturm, concerning the $\log$-removal. 

\item The second major case will be addressed by, 
after some technical treatment,  
 embedding our operator into a complex family of operators. 
 Residing on one side of this family is a particular case of the $L^2$-estimates for the damped oscillatory integral operators
 by Phong-Stein \cite{PS97-2}.
We will quote their results directly. 
The other side is $H_E^1 \to L^ 1$ estimates more {\it a la} Phong-Stein \cite{PS86} and Pan \cite{PAN91, PAN95} (see \cites{RS86} for earlier related results and \cites{GS99, SY14, Y04} for recent application). 
Although Pan established such estimates for polynomial phases in all dimensions, 
they do not apply directly to our setting due to  
an extra truncation coming from the resolution algorithm. 
A complete proof will appear in Section \ref{major2}.

\end{enumerate}

\subsection*{Notation:} 
We use $X\lesssim Y$ to mean ``there exists a constant $C>0$ such that $X\leq CY $'', where in this context (and throughout the entire paper) constants $C$ may depend on the phase $S$ and the cut-off function $\chi$ or $\phi$ but must be independent of the parameter $\lambda$ and the functions $f$ and $g$. 
The expression $X\gtrsim Y$ is analogous,
and $X\sim Y$ means both $X\lesssim Y$ and $X\gtrsim Y$. 
We will also use a (rectangular) box to mean a rectangle whose sides are parallel to the coordinates.

%
%
%
%
%
%
%
%
%
%
%
%
%
%
%
\section{Preliminary: analytic lemmata}\label{pre01}
The purpose of this section is to provide the analytic tools to prove our main theorem.
Rigorous statements of (A.1-A.3) or (B.1-B.3) are given as follow: 
\begin{lemma}\label{size1}
Let $T_\lambda f$ be given in \eqref{mo} and suppose $\supp\chi$ is contained in a $\sigma\times \rho$ rectangular box, then for $1 \leq q\leq \infty$,   
$$
|\langle T_\lambda f, g\rangle | \leq  \sigma^{\f1{ q}}\rho^{\f 1 {q'}} \|f\|_q\|g\|_{q'}.  
$$
\end{lemma}
This lemma can be interpreted as variants of Schur's test, or proved by H\"older's inequality or by interpolations. We omit the details. 
The second one is Phong-Stein's operator van der Corput, whose proof can be found in many places; see \cite{PS94-1, PS94-2, RY01, GR04}. 
Indeed, it can be obtained by the proof of Lemma \ref{L2} in Section \ref{major1}.  
\begin{lemma}\label{os2}
Let  $\phi(x,y)$ be a smooth function supported in a strip of $x$-width and $y$-width
no more than $\delta_1$ and $\delta_2$ respectively. Assume also 
\begin{align}\label{ps1}
|\partial_y \phi(x,y)| \lesssim \delta_2^{-1} \q\q\mbox{and} \q\q|\partial^2_y \phi(x,y)| \lesssim \delta_2^{-2}. 
\end{align}
Let $\mu>0$ and $S(x,y)$ be a smooth function s.t. for all $(x,y)\in \supp(\phi)$:
\begin{align}\label{PS1}
|\partial_x\partial_yS(x,y)| \gtrsim \mu
\q\q\mbox{and}\q\q
 |\partial_x\partial_y^\alpha S(x,y)| \lesssim \frac{\mu}{\delta_2^{\alpha}} \q\mbox{for}\q \alpha=1,2\,.
\end{align}
Then the operator defined by 
\begin{align}
T_\lambda f(x)  =\int e^{i\lambda S(x,y)}f(y)\phi(x,y)dy 
\end{align}
satisfying 
\begin{align}\label{PS2}
\|T_\lambda f\|_2 \lesssim |\lambda\mu|^{-\frac 1 2 }\|f\|_2.
\end{align}
\end{lemma}
The third one is the almost-orthogonality principle, which can be obtained by H\"older's inequality (Cauchy-Scharz's inequality).
\begin{lemma}\label{or3}
Let $\{T_h\}$ be a sequence of operators defined as in \eqref{mo} with $\chi_h$ as the smooth cut-offs. Suppose for every $y$ 
$$
\sum_{h}\int |\chi_h(x,y)| dx\lesssim 1 
$$
and the analogue also holds for every $x$. Then for $ 1\leq q \leq \infty$,  
 $$
 \sum_h |\langle T_h f, g\rangle | \lesssim \sup_h  \|T_h f\|_{q} \|g\|_{q'}.  
 $$
\end{lemma}
As usual, they are referred as size-estimates, oscillation-estimates and the almost-orthogonality respectively. Coupled with the resolution algorithm, they are sufficient to settle the minor case. 

We also need the following lifting trick to address the first major case.  
The oldest reference we can find is Zygmund's book \cite{ZY}; see also \cite{HP92, SY14}. 
\begin{lemma}\label{INT001}
Let $p_0>1$ and $H$ be a sublinear operator such that 
\begin{align*}
\|H(f)\|_{L^{p_0}({\mathbb R)}\to L^{p_0}{(\mathbb R)}} \leq A
\end{align*}
and
\begin{align*}
\|H(f)\|_{L^1{(\mathbb R)}\to L^\infty{(\mathbb R)}} \leq B
\end{align*}
Then for $1<p\leq p_0$ one has
\begin{align}\label{in}
\int_{\mathbb R} |Hf(x)|^p|x|^{{(p-p_0)}/{(p_0-1)}} dx \leq C^{p} \int _{\mathbb R}|f(x)|^pdx
\end{align}
where $C\sim B^{\frac{p_0/p-1}{p_0-1}}A^\frac{p_0-p_0/p}{p_0-1}$. 
\end{lemma}
We will use this lemma to lift $L^2\to L^2$ mappings to $L^p\to L^p$ mappings.
\begin{proof}[Proof of Lemma \ref{INT001} \cite{SY14, HP92}]
Let $ \tilde H(f) (x) = |x|^{1/(p_0-1)} Hf(x) $ and a measure $d\mu(x)  = |x|^{-p_0/(p_0-1)}d x$.
One can see that 
\begin{align}\label{IN1}
 \|\tilde H\|_{L^{p_0}(dx)\to L^{p_0}(d\mu)} \leq A.
\end{align}
We also have
\begin{align*}
 |\tilde Hf(x)| \leq |x|^{1/(p_0-1)} \|Hf\|_{L^\infty(dx)} \leq   B  |x|^{1/(p_0-1)} \|f\|_{L^1(d x)}.
\end{align*}
For every $\lambda >0$,  
\begin{align*}
\mu( \{x\in\mathbb R: |\tilde Hf(x)|>\lambda\}) 
&\leq \mu( \{x\in\mathbb R:B  |x|^{1/(p_0-1)} \|f\|_{L^1(d x)}>\lambda\})
\\
&
= \mu( \{x\in\mathbb R:|x|>\big(\lambda \cdot (B \|f\|_{L^1(d x)})^{-1}\big)^{p_0-1})\}
\\
&=\int _{\lambda \cdot (B \|f\|_{L^1(d x)})^{-1})^{p_0-1}} |x|^{-p_0/(p_0-1)}dx
\\
&= \frac {2}{p_0-1} B \|f\|_{L^1(d x)}/\lambda
\end{align*}
and thus 
\begin{align}\label{IN2}
\|\tilde H \|_{{L^1(d x)}\to L^{1,\infty}(d\mu ) } \leq  \frac {2}{p_0-1} B.
\end{align}
Interpolating (\ref{IN1}) and (\ref{IN2}) yields (\ref{in}) with $C \sim B^{\frac{p_0/p-1}{p_0-1}}A^\frac{p_0-p_0/p}{p_0-1}$.
\end{proof}

%
%
%
%
%
%
%
%
%
%
%
%
%
%
%

\section{Preliminary}\label{res}
In the previous section, we introduced the main analytic tools for the proof of the main theorem. 
A crucial step in applying these tools is to establish useful lower bounds of $|S''_{xy}|$, which is the objective of the current section. 
 More precisely, we employ the resolution algorithm developed in \cite{X2013} to decompose 
a neighborhood of the origin (assuming $S''_{xy}(0,0)=0$) into finitely many regions on which
$S''_{xy}$ behaves like a monomial. This section is essentially the same as Section 3 in \cites{GX15}.

To begin with, we introduce the following concepts that are related to the Newton polygon.  
Let $P(x,y)$ be a real analytic function defined on some neighborhood of the origin. In the proof of our main theorem,
we will take $P =S''_{xy}$.  
Write the Taylor series expansion of $P$ as 
\begin{align*}P(x,y)  = \sum\limits_{p,q\in \mathbb N}c_{p,q}x^py^q
\end{align*}
and drop all the terms with $c_{p,q} =0$ in the above expression.
The Newton polygon (not the reduced Newton polygon) of $P$, denoted by $\mathcal N(P)$,  is the 
convex hull of the union of $[p,\infty)\times [q,\infty)$ for all
$(p,q)$ with $c_{p,q} \neq 0$.  
The Newton diagram $\mathcal D(P)$ is the boundary of $\mathcal N(P)$.    
This diagram consists of two non-compact edges, a finite (and possibly empty) collection of compact edges $\mathcal E(P)$, and a finite collection of vertices $\mathcal V(P)$. The vertices and the edges are called the faces of the Newton polygon and the set of faces is denoted by $\mathcal F(P)$. The vertex that lies on the bisecting line $p=q$, or if no such vertex exists, the edge that intersects $p =q$ is called 
the main face of $\mathcal N(P)$.
For $F\in \mathcal F(P)$, define 
\begin{align}\label{edgeP}
P_F(x,y) = \sum\limits_{(p,q)\in F} c_{p,q} x^p y^q.
\end{align}
The set of supporting lines of $\mathcal N(P)$, denoted by $\mathcal {SL}(P)$, are the lines that intersect the boundary of $\mathcal N(P)$ and do not intersect any other points of $\mathcal N(P)$. 
Notice that each supporting line contains at least one vertex of $\mathcal N(P)$ and each edge of $\mathcal N(P)$ lies in exactly one supporting line. Thus, 
we will also identify an edge with the supporting line containing it.  
There is a one-to-one correspondence $\mathcal M$ between $\mathcal {SL}(P)$ and the set of numbers $[0,\infty]$, given by defining
$\mathcal M(L)$ to be the negative reciprocal of the slope of $L$ for each $L\in \mathcal {SL}(P)$. 
We will often use the notation $L_m\in \mathcal {SL}(P)$ to refer the supporting line with $\mathcal M (L_m)= m$.

\begin{theorem}\label{rs90}
For each analytic function $P(x,y)$ defined in a neighborhood of $0\in\mathbb R^2$, 
there is an $\epsilon>0$ such that, up to a measure zero set, one can 
partition $U= (0,\epsilon)\times (-\epsilon,\epsilon)$ into a finite collection of `curved triangular' regions
$\{
U_{n,g, {\bsb \alpha}, j}
\}
$
on each of which $P(x,y)$ behaves like a monomial in the following sense. 
For each $U_{n,g, {\bsb \alpha}, j}$, there is change of variables 
\begin{align}\label{CV01}
\rho_n^{-1}(x_n,y_n)=(x,y) 
\end{align}
given by 
\begin{align}\label{change08}
\begin{cases}
 x_n= x 
 \\ 
 y_n= \gamma_n(x) +y_n   x^{m_0+\dots+m_{n-1}}  \end{cases}
\end{align}
with $\gamma_n(x)$ either given by a convergent fractional power series  
 \begin{align}\label{gamma1x}
 \gamma_n(x) = \sum\limits_{k=0}^{\infty} r_kx^{m_{0}+m_{1}+\cdots+m_k},
 \end{align}
or  a polynomial of some fractional power of $x$
 \begin{align}\label{gamma2x}
 \gamma_n(x) = \sum\limits_{k=0}^{n-1} r_kx^{m_{0}+m_{1}+\cdots+m_k},
 \end{align}
such that for any pre-seclected $K\in \mathbb N$, for all $0\leq a, b\leq K$ and $(x,y)\in U_{n,g, {\bsb \alpha}, j}$ one has
\begin{align}
&|P_n(x_n,y_n)| \sim |x_n^{p_{n}}y_n^{q_{n}}| \label{key1}
\\
&|\partial_{x_n}^{a} \partial_{y_n}^{b}P_n(x_n,y_n)| \lesssim  \min\{1, | x_n^{p_{n}-a}y_n^{q_{n}-b} \}\label{key2}|
\end{align}
and
\begin{align}\label{new}
| \partial_{y}^{b}P(x,y)|  \lesssim \min\{1, |x^{p_{n} - b(m_0+\dots +m_{n-1})}y_n^{q_{n}-b}|\}.
\end{align}
Here $P_n$ is defined by
$$
P_n(x_n,y_n) = P (\rho_n^{-1} (x_n,y_n)) = P(x,y)
$$
and $(p_n,q_n)$ is some vertex of the Newton polygon of $P_n$. 
\end{theorem}
The reason we phrased this theorem only in the right-half plane is to avoid writing the absolute value of $x$. Indeed, 
by changing $x$ to $-x$, the above theorem also applies to $U=(-\epsilon, 0)\times (-\epsilon, \epsilon)$, consequently to $U= (-\epsilon,\epsilon)\times (-\epsilon,\epsilon)$ and any its open subset. 
See \cite{X2013} for a proof of this theorem; below we provide only a sketch of the theorem for the purpose of outlining some useful ideas and introducing some terminology for later sections. In particular, we will need to further decompose each $U_{n,g, {\bsb \alpha}, j}$ into ``rectangular boxes''.

\subsection{A sketch of the proof}

Let $U$ be as in the theorem above. 
Set 
\begin{align*}
\begin{cases}
U_0 = U,
\\
P_0 = P,
\\
(x_0,y_0) = (x,y),
\end{cases}
\end{align*}
which can viewed as the original input for the algorithm below. 
The subindex 0 is used here to indicate the 0-th stage of the iteration. At each stage, iterations are always performed on some triple $[U, P, (x,y)]$. Here $P$ is a convergent (fractional) power series 
 in $(x,y)$, which are local coordinates centered at the origin, and  $U=(0,\epsilon)\times (-\epsilon,\epsilon)$ with $\epsilon>0$ being a small number depending on $P$. For convenience, we refer to such a triple $[U, P, (x,y)]$ as a standard triple. 

Choose one supporting line $L_m\in \mathcal {SL}(P)$ which contains at least one vertex $(p_0,q_0)=V\in\mathcal V(P) $. 
Let $E_l$ and $E_r$ be the two edges on the left and right of $V$, respectively. Set $m_l =\mathcal M(E_l)$ and $m_r =\mathcal M (E_r)$. Then $0\leq m_l\leq m \leq  m_r\leq \infty$. 
We will consider the region $|y|\sim x^m$ in each of the following three cases: 
\textbf{Case (1).} $m_l<m<m_r$, 
\textbf{Case (2).} $m=m_l$ and 
\textbf{Case (3).} $m=m_r$. 
In \textbf{Case (1)}, we have informally that the vertex $V$ `dominates'  $P(x,y)$, while in \textbf{Case (2)} and \textbf{Case (3)} we have that $E_l$ dominates $P(x,y)$ and $E_r$ dominates $P(x,y)$,  respectively.

In \textbf{Case (1)}, $p_0+m{q_0} < p+mq$ for any other $(p,q)$ with $c_{p,q} \neq 0$. 
Thus in the region $|y|\sim x^m$, when $|x|$ it sufficiently small, the monomial 
$$
|P_V(x,y)|=|c_{p_0,q_0}x^{p_0}y^{q_0}|\sim |x^{p_0+mq_0}|
$$ 
is the dominant term in $P(x,y)$, since 
$$
|P(x,y)-P_V(x,y)|=O(x^{p_0+mq_0+\nu}) \q \textrm {for some} \q\nu>0.
$$
 Thus 
\begin{align}
P(x,y)\sim P_V(x,y)=c_{p_0,q_0}x^{p_0}y^{q_0}
\end{align}
and we can make 
$$
\frac {|P(x,y)-P_V(x,y)|}{|P_V(x,y)|}
$$
to be arbitrarily small by choosing $\epsilon $ sufficiently small, i.e., by shrinking the region $U$.  
Moreover, for any pre-selected $a, b\geq 0$, 
\begin{align}\label{DE01}
|\p_x^a\p_y^b P(x,y)|\lesssim \min\{1, |x^{p_0-a} y^{q_0-b}|\}.
\end{align}
We refer such a region as a `good' region defined by the vertex $V$ and denote it by $U_{0,g,V}$. 
\begin{center}
\begin{tikzpicture}[scale=0.8]
\draw [<->,thick] (0,9) node (yaxis) [above] {$y$}
        |- (9,0) node (xaxis) [right] {$x$};
  \draw[help lines] (0,0) grid (8,8);
     \fill[gray!30] (2,8)--(2,4)--(3,2)--(5,1)--(8,1)--(8,8);
       \draw[thick] (5,0)--node[below left] {$l$}(0,5) ;
       \fill (5,1)   circle (2pt) (3,2) circle (2pt) (2,4)  circle (2pt);
        \draw [thick]	(9,1)--
        			(5,1) node[below left] {$V_3$}--
			(3,2)node[below left] {$V_2$}--
			(2,4)node[below left] {$V_1$}--
			(2,9);
	\node at (5,5) {Newton polygon};
	\node at (2.7,4){(2,4)};
	\node at (3.5,2.3) {(3,2)};
	\node at (5,1.5) {(5,1)};
	\node at (11,6) {$P(x,y) = x^5y-x^3y^2+x^2y^4$};
	\node at (10.2,5) {$P_{V_2}(x,y) = -x^3y^2$};
	\node at (11,4) {$|x|^{2}\lesssim |y| \lesssim |x|^{1/2}$};
	\node at (7,9.5) { \emph{Case (1): The vertex} $V_2$\emph{ is dominant, where} $1/2< m < 2$};
	\node at (4.5,-1) {Figure 1.};
\end{tikzpicture}
\end{center}

\begin{center}
\begin{tikzpicture}[scale=0.8]
\draw [<->,thick] (0,9) node (yaxis) [above] {$y$}
        |- (9,0) node (xaxis) [right] {$x$};
  \draw[help lines] (0,0) grid (8,8);
     \fill[gray!30] (2,8)--(2,4)--(3,2)--(5,1)--(8,1)--(8,8);
       \fill (5,1)   circle (2pt) (3,2) circle (2pt) (2,4)  circle (2pt);
        \draw [thick]	(9,1)--
        			(5,1) node[below left] {$V_3$}--
			(3,2)node[below left] {$V_2$}--
			(2,4)node[below left] {$V_1$}--
			(2,9);
	\node at (5,5) {Newton polygon};
	\node at (2.7,4){(2,4)};
	\node at (3.5,2.3) {(3,2)};
	\node at (5,1.5) {(5,1)};
	\node at (11,6) {$P(x,y) = x^5y-x^3y^2+x^2y^4$};
	\node at (11,5) {$P_{V_1V_2}(x,y) = -x^3y^2+x^2y^4$};
	\node at (12,4) {$|y|\sim |x|^{1/2}$};
	\draw[thick] (0,8)--(4,0);
	\node at (6,9.5) { \emph {Case(2): The edge} $V_1V_2$ \emph{ is dominant, where} $ m =1/ 2$};
		\node at (4.5,-1) {Figure 2.};

\end{tikzpicture}
\end{center}

 \textbf{Case (2)} and \textbf{Case (3)} are essentially the same, but are much more complicated than \textbf{Case (1)}.  We focus on \textbf{Case (2)}. 
 Set  $m_0 =m_l$, $P_0 =P$ and $E_0 =E_l$. One can see that $p_0+m_0q_0 =p+m_0q$ for all $(p,q)\in E_0$ and $p_0+m_0q_0 <p+m_0q$
  for all $(p,q) \notin E_0$ and $c_{p,q}\neq 0$. 
 Then for all $(p,q)\in E_0$, 
 $|x^{p}y^q|\sim |x^{p_0}y^{q_0}|$ in the region $|y| \sim |x|^{m_0}$. Set $y  = rx^{m_0}$ and let $\{r_*\} $ be the set of nonzero roots of $ P_{0,E_0}(1,y)$ with orders $\{s_*\}$.  
Let $\rho_0>\epsilon$ be a small number chosen so that the $\rho_0$ neighborhoods of 
all the roots $r_*$ will not overlap. 
 If $|r-r_*| \geq\rho_0$ for all roots $r_*$ and if $\epsilon$ sufficiently small,  then  
 \begin{align}\label{cool}
|P(x,y)|= |P_0(x,y)| \sim_{\rho_0} |P_{0,E_0}(x,y)|\sim |x^{p_0}y^{q_0}|\sim x^{p_0+m_0{q_0}},
 \end{align}
which means the edge $E_0$ dominates
 $P_0(x,y)$. Similarly, for any preselected $a, b \geq 0$, (\ref{DE01}) still holds.  
 We refer such regions as `good' regions defined by the edge $E_0$ and denote them by $U_{0, g, E_0, j}$'s. 
It is of significance to observe that one can enlarge each $U_{0, g, E_0, j}$ by decreasing the value of $\rho_0$, while 
$|P(x,y)|\sim  |x^{p_0}y^{q_0}|$ still holds but with the new implicit constant depending on the new $\rho_0$. 
However, one should not worry about the dependence on those $\rho_0$, for they will be chosen to 
rely only on the original function $P$. We will use this observation to further decompose  $U_{0, g, E_0, j}$ into rectangular boxes later.

 For each root $r_0\in \{r_*\}$, there is an associated `bad' region defined by $y=rx^{m_0}$ where $|r-r_0|< \rho_0$ and $0<x<\epsilon$.
 We say that this is a bad region defined by the edge $E_0$ (and the root $r_0$).
 Each bad region is carried to the next stage of iteration via change of variables. Set
 \begin{align*}
 \begin{cases}
 x =x_1,
 \\
 y =x_1^{m_0}(r_0+y_1).
 \end{cases}
 \end{align*}
 Notice that this bad region is contained in the set where $|y_1 |= |r-r_0| <\rho_0$, so that in $(x_1,y_1)$ coordinates, the bad region is now $0<x_1<\epsilon$ and $-\rho_0<y_1<\rho_0$. 
Set
 \begin{align*}
 P_1(x_1,y_1) = P_0(x_1, r_0 x_1^{m_0} +y_1x_1^{m_0}), 
 \end{align*}
 and for any choice of $\epsilon_1 > \rho_0$, let
 \begin{align*}
 U_1 =\{(x_1,y_1): 0<x_1<\epsilon_1, |y_1|< \epsilon_1 \}.
 \end{align*}
 Then $[U_1, P_1, (x_1,y_1)]$ is a standard triple and same arguments can be applied again. 
Notice that the edge $E_0$ in $\mathcal N(P)$ is ``collapsed'' into the leftmost vertex of $\mathcal N(P_1)$. More precisely,  
 if $(p_{1,l}, q_{1,l}) $ is the leftmost vertex of $\mathcal N(P_1)$, then 
 \begin{align}\label{key88}
(p_{1,l}, q_{1,l}) = (p_0 +m_0q_0, s_0).
 \end{align}
 This is a crucial bookkeeping identity for understanding the behavior of $P(x,y)$ in later stages of the iteration. 
Now we repeat the previous arguments. If either a vertex or an edge is dominant, then 
\begin{align*}
|P_1(x_1,y_1)| \sim |x_1^{p_1} y_1^{q_1}|
\end{align*} 
 for some vertex $(p_1,q_1)$ of $\mathcal N(P_1)$. Otherwise
there is an edge $E_1\in\mathcal E(P_1)$ and a non-zero root $r_1$ of $P_{1, E_1}(1, y_1)$ together with a neighborhood 
$(r_1 -\rho_1,r_1+\rho_1)$ that define a bad region. 
 Change variables to
  \begin{align*}
 \begin{cases}
 x_1 =x_2
 \\
 y_1 =x_2^{m_1}(r_1+y_2)
 \end{cases}
 \end{align*}
 and set 
 \begin{align*}
 P_2(x_2,y_2)  = P_1( x_2, x_2^{m_1}(r_1+y_2)).
 \end{align*}
Now we iterate the above argument. In the $k$-th stage of iteration, if either a vertex or an edge is dominant, then
\begin{align}\label{ESk}
|P(x,y)| = |P_k(x_k,y_k)|\sim |x_k^{p_k}y_k^{q_k}|
\end{align}
and 
\begin{align}\label{DEk}
|\p_{x_k}^a\p_{y_k}^b P_k(x_k,y_k) |\lesssim \min\{ 1, |x_k^{p_k-a}y_k^{q_k-b}| \}.
\end{align}
Here $(p_k,q_k)$ is a vertex of $\mathcal N(P_k)$ and 
\begin{align*}
\begin{cases}
x_{j-1} = x_{j},
\\
y_{j-1} = (r_{j-1}+y_j)x^{m_{j-1}},
\end{cases}
\end{align*}
for $1 \leq j\leq k$, i.e.,
\begin{align*}
\begin{cases}
x = x_0 =\dots = x_k
\\
y = y_0 = r_0x_0^{m_0}+r_1x_0^{m_0+m_1}+\dots +r_{k-1}x_0^{m_0+\dots+m_{k-1}}+y_kx^{m_0+\dots +m_{k-1}}.
\end{cases}
\end{align*}
The iterations above form a tree structure and each branch of this tree yields a chain of standard triples  
\begin{align}\label{chain70}
[U_0,P_0,(x_0,y_0)]\to[U_1,P_1,(x_1,y_1)]\to\dots\to [U_k,P_k,(x_k,y_k)]\to\cdots.
\end{align}
One may hope each such chain is finite, unfortunately this is not necessarily the case. For those that terminate after finitely many steps, say at the $n$-th stage of iteration, it must be the case that no bad regions are generated by $[U_n,P_n, (x_n,y_n)]$, and so we obtain Theorem \ref{rs90} with the corresponding change 
of variables given by (\ref{gamma2x}).  

For each of those branches that does not terminate,
one can show that (see Lemma 4.6 \cite{X2013}), there exists  $n_0\in\mathbb N$ such that for $n\geq n_0$, 
 $\mathcal N(P_n)$ has only one compact edge $E_n$ and $P_{n,E_n}(x_n,y_n)= c_n(y_n-r_nx_n^{m_n})^{s_{n_0}}$,
 where $s_{n_0}$ is the order of the root $r_{n_0}$. On such a branch, one should instead perform  
 the following change of variables:
 $$
y_{n_0} =  y_{n_0+1}x_{n_0}^{m_{n_0}}+ \sum\limits_{k=n_0}^{\infty} r_kx_{n_0}^{m_{n_0}+m_{n_0+1}+\cdots+m_k}.
 $$
Under this alternate change of variables, iteration along this branch stops immediately at stage $n_0+1$. This corresponds to the change of variables in (\ref{gamma1x}). 
Since the total number of branches in the iteration is bounded above (see Lemma 4.5  in \cite{X2013}) the modified algorithm fully terminates after finitely many steps. In particular, the total number of good regions are also finite. 

We now give some explanation of the subindices $({n,g, {\bsb \alpha}, j})$ in Theorem \ref{rs90}. 
The letter `$n$' represents the stage of iterations, `$g$' indicates the region is `good',  `${\bsb \alpha}$' contains the information necessary to make the change of variables 
from $[U_0,P_0,(x_0,y_0)]$ to a specific $[U_n,P_n,(x_n,y_n)]$ and `$j$' is used to list the all the `good' regions generated by  $[U_n,P_n,(x_n,y_n)]$. 
The cardinality of the tuples $({n,g, {\bsb \alpha}, j})$ is finite and depends on the original function $P$.
We often use $U_{n,g}$ to represent  $U_{n,g}= U_{n, g, \bsb\alpha,j}$ for some $\bsb\alpha$ and some $j$. 

The identity (\ref{key88}) is very useful to help estimate $P(x,y)$ in higher stages of the iteration. Indeed, 
consider the $n$-th stage of iteration and let $U_{n, g, {\bsb \alpha}, j}$ be a good region. Then 
$|P(x,y)| =|P_n(x_n,y_n)| \sim |x_n^{p_n} y_n^{q_n}|$, for some vertex $V_n=(p_n,q_n)$ of $\mathcal N(P_n)$. 
For $0 \leq j \leq n-1$,
let $(p_{j, l},q_{j,l})$ be the leftmost vertex of $\mathcal N(P_j)$, and furthermore let $(p_j,q_j)$ be the left vertex of the
edge $E_j$ and let $r_j$ be the root that governs the next stage of iteration.
Assume $s_j$ is the order of the root  $r_j$. 
 We claim that for $0\leq j\leq n-1$, 
\begin{align}\label{ind1}
\begin{cases}
p_{j+1,l} = p_{j} +q_{j}m_j , 
\\
 q_{j+1,l}  =s_j \leq q_j,
\\
p_{j}+q_jm_j \leq p_{j,l}+q_{j,l}m_j,
\\
\end{cases}
\end{align}
and if $L_{m_n}$ is a supporting line of $\mathcal N(P_n)$ through $V_n$, then 
\begin{align}\label{IT5}
p_{n}+m_n q_n \leq p_{n,l}+m_n q_{n,l}  \leq  p_0+m_0q_0 +s_0\sum_{1\leq j \leq n-1} m_j.
\end{align}
Indeed the first two identities in (\ref{ind1}) are of the same nature as the one in (\ref{key88}). 
The third one comes from the fact that the vertex $(p_{j,l},q_{j,l})$ lies on or above $E_j$, which is the
 supporting line through $(p_j,q_j)$ of slope $-1/m_j$. 
Iterating (\ref{ind1}) yields (\ref{IT5}).

Now we come to the crucial observation which is used to deduce
 structural information of the phase for the exceptional cases.  
In Section \ref{minor}, 
we will see that 
such exceptional cases occur 
in some good region
when $(p_n, q_n) = (p_{n,l}, q_{n,l})$, $p_0=0$ and $q_n =q_{0,l}$.  
Together with the second identity of \eqref{ind1}, they imply $q_{0,l}= q_0 =s_0 = q_{1,l} =q_1 =s_1 = \cdots  s_{n-1}=q_{n,l}=q_n$.  
Consequently, for $0\leq j \leq n-1$, the Newton polygon of $P_j$ has only one compact edge $E_j$ and the restriction of $P_j$ to this edge is 
$$
P_{j, E_j} (x_j, y_j)= c_jx_j^{p_j}(y_j -r_jx_j)^{s_j} = c_jx_j^{p_j}(y_j -r_jx_j)^{q_{0,l}}. 
$$
Therefore in this good region 
\begin{align}\label{cru99}
P(x,y) = P_n(x_n,y_n) \sim x_n^{p_n}y_n^{q_n} \sim (y-\gamma_n(x))^{q_{0,l}}. 
\end{align}

\subsection{A smooth partition}
Let $U$ be a small neighborhood of the origin such that one can apply Theorem \ref{rs90} to it. 
 Divide $U$ (up to a set of measure zero) into four different regions: $U_E$, $U_N$, $U_W$ and $U_S$
 (representing the east, north, west and south portions), defined by 
\begin{align}
\begin{cases}
U_E= \{(x,y)\in U: x>0, -Cx < y < Cx\}
\\
U_W = -U_E
\\
U_N = \{(x,y)\in U: C |x| < y \}
\\
U_S= -U_N.
\end{cases}
\end{align}
We choose the above constant $C>0$ such that $2^{-10}C$ is greater than the absolute value of any root of 
$P_E(1,y)$ or $P_E(-1,y)$, where $E$ is the edge of slope $-1$. 
We shall then define smooth functions $\phi_E$, $\phi_N$, $\phi_W$ and $\phi_S$ whose supports are contained in 
$\frac 2 3 U$ and such that for any function $\Psi$ supported in $\frac 1 2 U$, the
following holds :
\begin{align*}
\Psi(x,y) = (\phi_E(x,y)+\phi_N(x,y)+\phi_W(x,y)+\phi_S(x,y))\Psi(x,y) \ \ {\rm for}\,\, (x,y)\neq 0.   
\end{align*} 
Moreover, $\phi_E$ is essentially supported in $U_E$, in the sense that $\supp \phi_E\subset  U^*_E$, where 
\begin{align}
U_E^* = \{(x,y)\in U: -2Cx <y <2C x\},
\end{align} 
and similarly for $\phi_N$ and so on.  

In what follows, we focus on $\phi_E$ and $U_E$. Similar results for $U_W$, $U_N$ and $U_S$ can be obtained
by changing $x$ to $-x$, switching $y$ and $x$, and making both changes, respectively. 
Notice $|y|\lesssim |x|$ for $(x,y)\in U_E^*$, and by shrinking $U$ if necessary, one has $m_0 \geq 1$ if $U_{n,g,{\bsb\alpha},j}$ has nonempty intersection with $U_E^*$, where $m_0$ is the exponent of the first term of $\gamma_n(x)$ in (\ref{gamma1x}) or (\ref{gamma2x}). 
We will partition $\phi_E$ into a sum of smooth functions $\phi_\mathcal R$
such that each $\phi_\mathcal R$ is essentially supported in a box $\mathcal R$ which is essentially contained in one $U_{n,g,{\bsb\alpha},j}$.  
Consequently $P(x,y)$ still behaves like a constant in $\supp \phi_\mathcal R$. 
We briefly describe how to obtain such partition and refer the readers to \cite{GX15} for details. First fix one good region $U_{n,g}$, 
which is contained in $x_n^{m_n'}\lesssim y_n  \lesssim x_n^{m_n}$ and in which $|P_n(x_n,y_n)|\sim |x_n^{p_n}y_n^{q_n}|$. 
Let $\sigma$ and $\rho$ be dyadic numbers and set $U_{n,g}(\sigma,\rho)$ be a subset of $U_{n,g}$ with 
 $x_n\sim \sigma$ and $y_n\sim \rho \sigma^{m_n}$. Notice $0<\sigma\lesssim 1$, $\sigma^{m_n'-m_n}\lesssim \rho \lesssim 1$
 and $P_{n}(x_n,y_n)$ is comparable to a fixed value in $U_{n,g}(\sigma,\rho)$.  
 Let $\Delta x$  (and $\Delta y$) be the length of a typical $x$-cross ($y$-cross res.) section of $U_{n,g}(\sigma,\rho)$. Then 
 \begin{align}
\begin{cases}\label{sim07}
&\Delta x \sim   \rho{\sigma ^{m_0+\dots +m_n} \cdot  \sigma^{1-m_0}}
\\
&\Delta y \sim \rho \sigma ^{m_0+\dots +m_n},
\end{cases} 
\end{align}
if $n\geq 1$ and 
\begin{align}
\begin{cases}\label{sim08}
&\Delta x \sim \sigma 
\\
&\Delta y \sim \rho\sigma ^{m_0} 
\end{cases} 
\end{align}
if $n=0$. We then cover $U_{n,g}(\sigma,\rho)$ by a collection of $\Delta x\times \Delta y$ boxes, denoted by $\{\mathcal R\}_{\mathcal R\in U_{n,g}(\sigma,\rho)}$. 
The collection operators associated to the boxes $\{\mathcal R\}_{\mathcal R\in U_{n,g}(\sigma,\rho)}$ will exhibit the orthogonality properties
(A.3). 
 Indeed, let $\pi_x$ and $\pi_y$ denote the othogonal projections of $\mathbb R^2$ onto the $x$- and $y$-axis, respectively. 
Then there exists a positive constant $C\in\mathbb N$
 independent of $\sigma$ and $\rho$ such that 
 \begin{align}\label{OT67}
\sum_{\mathcal R\in U_{n,g} (\sigma,\rho)} \Id_{\pi_x(\mathcal R)}(x)  \leq C \q 
\textrm{and}\q \sum_{\mathcal R\in U_{n,g} (\sigma,\rho)} \Id_{\pi_y(\mathcal R)}(y) \leq C. 
 \end{align}
Notice that $C$ can be also independent of each $U_{n,g}$, for the cardinality of the collection of good regions $U_{n,g,{\bsb\alpha}, j}$ is finite. 
 
To complete the smooth partition, let $\phi_\mathcal R$ be a non-negative 
smooth function essentially supported in $\mathcal R$, in the sense that it satisfies 
 \begin{align}
\mathcal R\subset  \supp \phi_{\mathcal R} \subset \tau \mathcal R 
 \end{align}
  and 
 \begin{align}
|\partial_x^a \partial_y ^b  \phi_{\mathcal R}(x,y)| \lesssim |\Delta x|^{-a} | \Delta y|^{-b}, \q \textrm {for} \q 0\leq a, b \leq 2.
 \end{align}
 Consequently,  
 \begin{align}
  \displaystyle  \sum_{\sigma}\sum_{\rho}\,\,\,\sum_{\mathcal R\in U_{n,g} (\sigma,\rho)} \phi_{\mathcal R} (x,y)\neq 0\,,\q \textrm {for all} \q(x,y)\in U_{n,g}.
 \end{align}
 Here $\tau >1$ is some fixed number independent of $\sigma$ and $\rho$ such that the dilation $\tau \mathcal R$ is essentially contained in 
 $U_{n,g}(\sigma,\rho)$, in the sense
 $|P(x,y)|=|P_n(x_n,y_n)| \sim   |x_n^{p_n}y_n^{q_n}|$ for $(x,y)\in \tau \mathcal R$.  
We can also control upper bounds of certain derivatives of $P(x,y)$ for $(x,y)\in \tau \mathcal R$. 
Indeed, by the chain rule, $ \p_{y}^b P(x,y)$ is equal to 
 \begin{align}
\p_{y_n}^b P(x,y) \left( \frac{ \p y} {\p{y_n}}\right) ^b =  \p_{y_n}^b P(x,y) \cdot  x^{-b(m_0+\dots+m_{n-1})}.
 \end{align}
By (\ref{DEk}), we have   
for $(x,y)\in \tau \mathcal R$, 
   \begin{align}
| \p_{y}^b P(x,y) | &\lesssim |x_n^{p_n}y_n^{q_n}| \cdot \Delta y ^{-b}  \sim   |P_n(x_n,y_n)| \cdot \Delta y ^{-b}. 
 \end{align}
 Similarly, viewing $y$ as a function of $x$, one has 
 \begin{align*}
 \p_{x} P(x,y) = (\p_{x} P)(x,y)+  (\p_{y} P)(x,y) \frac{ \partial y }{\p x}. 
 \end{align*}
 Since $|\frac {\p y}{\p x}| \lesssim |x|^{m_0-1} $,  $| \p_{x} P(x,y)| $ is bounded by 
  \begin{align*}
& |x_n^{p_n} y_n^{q_n}| |x_n|^{-1} + |x_n^{p_n}y_n^{q_n}| |x^{m_0+\dots+m_{n-1}+m_n}y_n|^{-1}\cdot  |x|^{m_0-1} 
 \end{align*}
 and it is obvious that the later is dominant. Thus $ |\p_{x} P(x,y)|\lesssim |x_n^{p_n}y_n^{q_n}| \Delta x^{-1 }$. 
Analogous calculations for $a =0, 1,$ and $2$ yield that   
  \begin{align*}
| \p_{x}^a P(x,y)| \lesssim  |x_n^{p_n}y_n^{q_n}| \cdot  \Delta x^{-a}\sim |P_n(x_n,y_n)|\cdot  \Delta x^{-a}
 \end{align*}
 for $(x,y)\in \tau\mathcal R$.  
 Finally, we associate the function 
 \begin{equation}
\dfrac {\phi_\mathcal R}{ \sum_{(n,g,{\bsb\alpha},j)}\sum_{\sigma}\sum_\rho\,\,\,\sum_{\mathcal R\in U_{n,g,{\bsb\alpha},j} (\sigma,\rho)} \phi_{\mathcal R} }\cdot 
 \phi_E \label{partou} 
 \end{equation}
to ${\mathcal R}$ (and, for convenience, redefine $\phi_{\mathcal R}$ to equal \eqref{partou}) and therefore 
 \begin{align}
  \phi_E(x,y) = \sum_{(n,g,{\bsb\alpha},j)}\sum_{\sigma}\sum_\rho\,\,\,\sum_{\mathcal R\in U_{n,g,{\bsb\alpha},j} (\sigma,\rho)} \phi_{\mathcal R}(x,y). 
 \end{align}

%
%
%
%
%
%
%
%
%
%
%
%
%
%
%
%
%
%

%
%
%
%
%
%
%
%
%
%
%
%
%
%
%
%
%
%

\section{Reduction to a good region} \label{Red}
In the proof of Theorem \ref{COM}, we lose no generality in assuming $(x_0,y_0) = (0,0)$.  
Let $(k,l)\in\mathbb N^*$ be any vertex of the Newton polygon of the phase $S$ at the origin. Equivalently,  
$(k-1, l-1)$ is any vertex of the Newton polygon of $P(x,y) = S''_{xy}(x,y)$. 
The goal is to establish the following for each such $(k,l)$:
\begin{align}\label{goal77}
\|T_0f\|_{\f {k+l}{k}} \lesssim |\lambda|^{-\f 1 {k+l}}\|f\|_{\f{k+l}{k} }. 
\end{align}
Through out the rest of this paper, we let $(k,l)\in \mathbb N_*^2$ be fixed.  

As it was done in the previous section, we decompose $\phi$ into the sum of four functions supported in the east, north, west and south regions respectively, i.e.  
$$
\phi(x,y) =\phi_E(x,y) +\phi_N(x,y)+\phi_W(x,y)+\phi_S(x,y)\,\,{\rm for}\,\, (x,y)\neq 0 \,.
$$
The bilinear form can be then written as 
$$
\langle T_0f, \, g\rangle   = \mathcal B_E(f,g) +\mathcal B_N(f,g)+\mathcal B_W(f,g)+\mathcal B_S(f,g)   
$$
where 
\begin{align}\label{TR70}
\mathcal B_E(f,g) =\iint e^{i\lambda S(x,y)}f(y)g(x)\phi_E(x,y) dxdy,
\end{align}
and so on. We will focus on $ \mathcal B_E(f,g)$ and the others can be handled similarly. 
Apply the smooth partition of unity from the previous section built from the function $P(x, y) = S''_{xy}(x, y)$;
 we can write
\begin{align*}
\phi_E = \sum_{(n,g, {\bsb\alpha}, j)} \sum_{\sigma, \rho} \sum_{\mathcal R\in U_{n,g,{\bsb\alpha}, j}(\sigma,\rho)} \phi_{\mathcal R},
\end{align*}
where $\sigma$ and $\rho$ are dyadic numbers. Write 
$$
 \mathcal B_E(f,g) =\sum_{(n,g, {\bsb\alpha}, j)} \sum_{\sigma, \rho} \sum_{\mathcal R\in U_{n,g,{\bsb\alpha}, j}(\sigma,\rho)} 
 \mathcal B_{\mathcal R}(f,g)
$$
where $ \mathcal B_{\mathcal R}(f,g)$ is define as in (\ref{TR70}) with $\phi_E$ replaced by $\phi_\mathcal R$. 
Since the number of $  U_{n,g,{\bsb\alpha}, j}$ is finite, it suffices 
to prove the desired estimate in one $U_{n,g,{\bsb\alpha}, j}$. We will focus on $U_{n,g}= U_{n,g,{\bsb\alpha}, j}$ for some 
${\bsb\alpha}, j$. 
Write 
$$
 \mathcal B_{n,g}(f,g) = \sum_{\sigma, \rho} \sum_{\mathcal R\in U_{n,g}(\sigma,\rho)} 
 \mathcal B_{\mathcal R}(f,g). 
$$
Notice that for $(x,y) \in U_{n,g}$ one has 
\begin{align}
|S''_{xy}(x,y)| \sim |x_n^{p_n}y_n^{q_n}|
\end{align}
where
\begin{align}\label{yn}
\begin{cases}
x=x_0=\dots=x_{n}
\\
y =y_0 = \gamma_n(x)+y_nx^{m_0+\dots+m_{n-1}}.
\end{cases}
\end{align}
Here 
\begin{align}
\gamma_n(x) =\sum_{j=0}^* r_0x^{m_0+\dots+r_{j}}
\end{align}
with $*=n-1$ or $\infty$; see (\ref{gamma1x}) and (\ref{gamma2x}). 
Notice that $m_0\geq 1$ since $(x,y) \in \supp \phi _E$.  
Each $\mathcal R$ is a $\Delta x\times \Delta y$ rectangle. Let $\mathcal R_x $ and $\mathcal R_y$ denote the orthogonal projections of 
$\tau \mathcal R$ into the $x$- and $y$-axis respectively. Notice that $|\mathcal R_x| \sim \Delta x$ and $|\mathcal R_y| \sim \Delta y$. 
Lemma \ref{os2} implies 
 \begin{align}\label{OS01}
\| \mathcal B_\mathcal R(f,g) \|  \lesssim |\lambda|^{-\frac 1 2}  \left(\inf_{(x,y)\in \phi_\mathcal R}|P(x,y)|\right)^{-\frac 1 2}
 \left\|f\Id_{\mathcal R_y}\right\|_2\left\|g\Id_{\mathcal R_x}\right\|_2.
  \end{align}
Combining Lemma \ref{size1}, \eqref{OS01} and their interpolation, we have   
\begin{lemma}
For each $1\leq p \leq \infty$, we have 
\begin{align}\label{SZ02}
\| \mathcal B_\mathcal R(f,g) \|\lesssim   \Delta x^{\f 1 p } \Delta y^{\f 1 {p'}} 
 \left\|f\Id_{\mathcal R_y}\right\|_p\left\|g\Id_{\mathcal R_x}\right\|_{p'}.
\end{align}
For $2\leq p \leq \infty$ 
\begin{align}\label{OS02}
\| \mathcal B_\mathcal R(f,g) \|\lesssim  \Delta y ^{1-\f 2 p} 
|\lambda|^{-\frac 1 p}  \left(\inf_{(x,y)\in \phi_\mathcal R}|P(x,y)|\right)^{-\frac 1 p}
 \left\|f\Id_{\mathcal R_y}\right\|_p\left\|g\Id_{\mathcal R_x}\right\|_{p'}
\end{align}
and for $1\leq p \leq 2$
\begin{align}\label{OS02'}
\| \mathcal B_\mathcal R(f,g) \|\lesssim  \Delta x ^{1-\f 2 {p'}} 
|\lambda|^{-\frac 1 {p'}}  \left(\inf_{(x,y)\in \phi_\mathcal R}|P(x,y)|\right)^{-\frac 1 {p'}}
 \left\|f\Id_{\mathcal R_y}\right\|_p\left\|g\Id_{\mathcal R_x}\right\|_{p'}.
\end{align}
\end{lemma}
The estimates (\ref{OS02}) are employed for the $p\geq 2$ case and (\ref{OS02'})
for the $1\leq p\leq 2$ case.
Since  the arguments for these two cases are similar, 
we only focus on the $p\geq 2$ case, i.e. $l \geq k$.  
Assume also $(k,l) \neq (1,1)$ for otherwise the phase is non-degenerate and the desired estimates were already proved in \cite{HOR73}. 
Details of the proof will appear in the next three sections. In the rest of this section, we address the sharpness of Theorem \ref{COM}.  

In what follows, we assume we have established (\ref{goal77}) for all the vertices of $\mathcal N(S)$. 
For $(\f 1 {p\alpha}, \f 1 {p'\alpha})$ lying in the Newton diagram, 
we claim that $\|T_0\|_{p}\lesssim |\lambda|^{-\f 1\alpha }$ is sharp. 
For convenience, let us use $(a,b)$ to denote $(\f 1 {p\alpha}, \f 1 {p'\alpha})$. 
Let $\mathcal L_m$ be any supporting line through $(a, b)$. Then, one can find a small number $\delta>0$ depending only on $S$ such that 
$$
|\lambda S(\delta|\lambda|^{-\f {1}{a+bm}}, \delta|\lambda|^{-\f {m}{a+bm}})|\leq 2^{-10}.
$$   
If $f$ and $g$ are characteristic functions of $[0, \delta|\lambda|^{-\f {m}{a+bm}}]$ and $[0, \delta|\lambda|^{-\f {1}{a+bm}}]$ respectively,  then 
$$
\f {|\langle T_0f, g\rangle | }
{\|f\|_p\|g\|_{p'} }\sim \delta  |\lambda|^{-\f {1}{a+b}} = \delta  |\lambda|^{-\alpha}.
$$
Consequently, all the estimates associated to points in $\mathcal D(S)$ must be sharp, 
and those associated to the interior points of the $\mathcal N(S)$ can not be sharp, 
and those associated to the points outside $\mathcal N(S)$ can not be true.  

%
%
%
%
%
%
%
%
%
%
%
%
%
%
%
%
%
%

\section{The minor case}\label{minor}
In this section, we will exploit the three basic principles: size-estimates, oscillation-estimates and almost-orthogonality 
to address the minor case. 
In conjunction with the resolution algorithm, we will deduce certain structural information for the phase of the two major cases. 
The arguments are split into two parts according to $n=0$ and $n\geq 1$.

\subsection{
\bf{Case 1: n=0}} Then $ |P(x,y)| \sim |x^{p_0} y^{q_0}|$.  
Notice that $\Delta x \sim \sigma$ and $\Delta y \sim \rho \sigma^{m_0}$. 
For convenience in later arguments, let us define the quantities 
\begin{align*}
& A_{\sigma,\rho} =  \sigma^{1/p}(\rho\sigma^{m_0})^{1/p'}  
 \\
 &B_{\sigma,\rho} = |\lambda\sigma ^{p_0}(\rho^{q_0}\sigma^{m_0)}|^{-1/p} (\rho\sigma^{m_0})^{1 - 2/p}
\end{align*}
and their convex combination 
\begin{align*}
C_{\sigma,\rho}=  A^{\frac{k-1}k}_{\sigma,\rho}B^{\frac 1 k}_{\sigma,\rho}=|\lambda|^{-\frac 1 {k+l}}  \left (\frac {\sigma ^{k-1}(\rho\sigma^{m_0})^{l-1}}{\sigma ^{p_0}(\rho\sigma^{m_0})^{q_0}}\right)^{\frac 1 {k+l}} .
 \end{align*}
 By  (\ref{SZ02}), (\ref{OS02}) and their interpolation,  
 \begin{align}\label{LE01}
 |\mathcal B_\mathcal R(f,g)| 
 &\lesssim \min\{ A_{\sigma,\rho}, B_{\sigma,\rho}, C_{\sigma,\rho} \}  \left\|f\Id_{\mathcal R_y}\right\|_p\left\|g\Id_{\mathcal R_x}\right\|_{p'}.
 \end{align}
The fact $x^{p_0}y^{q_0}$ is the dominant term in $P(x,y)$ implies $C_{\sigma,\rho}\lesssim  |\lambda|^{-\frac 1 {k+l}}$. 

The number of $\mathcal R\in U_{0,g}(\sigma,\rho)$ is bounded by a constant since $\Delta x\sim \sigma$. Consequently, we do not need to worry about summing over $\mathcal R$ in this collection. To sum over $\sigma$ and $\rho$, consider first $ (k-1, l-1)\neq (p_0,q_0)$ and let 
 \begin{align*}
 \Theta_j = \{(\sigma,\rho): \sigma ^{k-1}(\rho\sigma^{m_0})^{l-1} \sim 2^{j} \sigma ^{p_0}(\rho\sigma^{m_0})^{q_0} \}.
 \end{align*}
 Again, by the fact $x^{p_0}y^{q_0}$ is the dominant term in $P(x,y)$, $ \Theta_j  =\emptyset $  if $j$ is larger than some fixed constant $C$. 
 The fact that $(k-1,l-1)$ and $(p_0,q_0)$ are linearly independent allows us to employ the almost-orthogonality Lemma \ref{or3} 
 to sum over $(\sigma,\rho)$ in the same $\Theta_j$ without any loss.  Thus 
 \begin{align*}
|\mathcal B_{n,g}(f,g)| & \leq \sum_{j\leq C} \sum_{(\sigma,\rho)\in\Theta_j} \sum_{\mathcal R\in U_{0,g}(\sigma,\rho)} |\mathcal B_{\mathcal R}(f,g)|   
 \\
 &\leq 
  \sum_{j\leq C}  \sum_{(\sigma,\rho)\in\Theta_j}   |\lambda |^{- \frac 1 {k+l}} 2^{\frac j {k+l}} 
  \left\|f\Id_{\mathcal R_y}\right\|_p\left\|g\Id_{\mathcal R_x}\right\|_{p'} \lesssim  |\lambda |^{- \frac 1 {k+l}}  \|f\|_p \|g\|_{p'}.
 \end{align*} 
 In the case $ (p_0,q_0) =   (k-1, l-1)$ and $k \neq 1$, 
 \begin{align*}
 B_{\sigma,\rho}  = |\lambda|^{- \frac{k } {k+l}} (\sigma^{-\frac{k-1} {k+l}} \rho^{-\frac{l-1}{k+l}})^k .
 \end{align*} 
And now we set  
  \begin{align*}
 \Theta_j = \{(\sigma,\rho): B_{\sigma,\rho}  \sim 2^j |\lambda|^{-\frac 1 {k+l}}  \}.
 \end{align*}
Notice that $l\geq k \geq 2$ and there is no loss when adding all $(\sigma,\rho)$ in the same $\Theta_j$ due to Lemma \ref{or3}.  
 One can then employ the $B_{\sigma,\rho}$ estimates for $j\leq 0$ and the $A_{\sigma,\rho}$ estimates for $j >0$ in the summation of $j$.   
 It is important that $k\neq 1$ in this case, for otherwise $C_{\sigma,\rho}$ is not a true convex combination of $A_{\sigma,\rho}$ and $B_{\sigma,\rho}$, and the above arguments will result in a $\log (2+|\lambda|)$ loss.  
 This case, i.e. our first major case,  $(p_0,q_0) = (0, l-1)$ (since $k = 1$) will be addressed in Section \ref{major1} by different approachs.

 \vspace{0.2in}
 
 \subsection{\bf{Case 2: $n\geq 1$. }}
 In this case, we need only focus on one single $\sigma$.  
Indeed, notice that $|y| \sim x^{m_0}$ and if $\sigma/\sigma' >C$ for some large constant, then  
$$
\{y: (x,y)\in \cup_{\rho}U_{n,g}(\sigma, \rho)\}\cap  \{y: (x,y)\in \cup_{\rho}U_{n,g}(\sigma', \rho)\} = \empty \,.
$$ 
By the almost-orthogonality Lemma \ref{or3}, it suffices 
to establish the desired bound for each single $\sigma$, given this bound is 
independent of $\sigma$.  
 Now for $(x,y)\in\mathcal R\in U_{n,g}(\sigma,\rho)$, one has 
 \begin{align*}
 \begin{cases}
 \Delta x \sim \rho\sigma^{m_0+\dots+m_n} \sigma^{1-m_0}
\\
 \Delta y \sim \rho\sigma^{m_0+\dots+m_n} 
 \end{cases}
 \end{align*}
 and 
 \begin{align*}
 |P(x,y)| \sim \sigma^{p_n +m_nq_n }\rho^{q_n}. 
 \end{align*}
 The estimates in (\ref{LE01}) hold with 
 \begin{align*}
& A_{\sigma,\rho} =  \rho\sigma^{m_0+\dots+m_n} \sigma^{(1-m_0)/p},
 \\
 &B_{\sigma,\rho} = |\lambda\sigma ^{p_n+m_nq_n}\rho^{q_n}|^{-1/p} (\rho\sigma^{m_0+\dots+m_n} )^{1 - 2/p}
\\
&
C_{\sigma,\rho}=  A^{\frac{k-1}k}_{\sigma,\rho}B^{\frac 1 k}_{\sigma,\rho}=|\lambda|^{-\frac 1 {k+l}}  
\sigma^{\f{\mu}{k+l}} \rho^{\f \nu {k+l}}
 \end{align*}
where 
\begin{align*}
\nu &= (k-1+l-1) -q_{n}
\\
 \mu &= (k-1 +l -1)(m_0+\dots m_{n}) +(1-m_0)(k-1) -  (p_n+m_nq_n).
 \end{align*}
Applying Lemma \ref{or3} to sum $\mathcal B_\mathcal R(f,g)$ over all $\mathcal R\in U_{n,g}(\sigma,\rho)$ gives 
$$
\sum_{\mathcal R\in U_{n,g}(\sigma,\rho)} |\mathcal B_\mathcal R (f,g)| \lesssim \min\{A_{\sigma,\rho},\, B_{\sigma,\rho}, C_{\sigma, \rho}\}\|f\|_p\|g\|_{p'}.
$$
We claim that $\mu$ and $\nu$ are both non-negative. An important observation is   
\begin{align}\label{LS0}
l-1 \geq s_0,
\end{align}
which implies $\nu \geq 0$ immediately. 
Indeed (\ref{LS0}) is obvious, if the vertex $(k-1,l-1)$ is equal to or lying on the left of $(p_0,q_0)$. 
 In this case if $\nu=0$, then one must have  
 \begin{align}\label{key78}
 \begin{cases}
 k=1
 \\
 l-1=q_0 =s_0 = q_n. 
 \end{cases}
 \end{align} 
 Otherwise the vertex $(k-1, l-1)$ lies on the right of $(p_0,q_0)$. Then  
\begin{align}\label{tem01}
s_0 \leq q_0 - (l-1) \leq (q_0 - (l-1))m_0 \leq (k-1)- p_0 \leq k -1 \leq l-1.  
\end{align}
We used the facts that $m_0\geq 1$ for the second inequality, that the vertex $(k-1, l-1)$ lies on or above the supporting line passing through $(p_0,q_0)$ with slope $-1/m_0$ for the third one, and that the assumption $l \geq k$ for the last one. 
In this case $\nu $ has to be positive, for if $\nu =0$ then $k=1$ and $s_0 =k-1 =l-1$, i.e.  $k=l=1$.  
This is the non-degenerate case excluded from the beginning.   

To show $\mu \geq 0$, we employ the induction formula \eqref{IT5}:
$$
p_n+m_nq_n \leq p_0+m_0q_0 +\sum_{j=1}^{n}q_{j,l}m_j 
$$
 and 
 $$
  p_0+m_0q_0 \leq (k-1)+m_0(l-1) . 
 $$
 Then
 $$
 \mu \geq \big((k-1) +m_0(l-1)- (p_0+m_0q_0)\big) + \sum_{j=1}^n((l-1)-q_{j,l})m_j\geq 0
 $$
and a necessary condition for $\mu =0$ is $l-1 =s_0 =q_n$.   
 
If $\nu >0$, then 
$
C_{\sigma,\rho} \lesssim  |\lambda|^{-\frac 1 {k+l}}\rho ^\nu. 
$
One can sum $C_{\sigma,\rho}$ over dyadic numbers $0<\rho<1$ and obtain a bound independent of $\sigma$.

Else $\nu =0$, (\ref{key78}) implies $k-1=0$, $l-1 = s_0 =q_0=q_n $, and thus $(p_n,q_n ) = (p_{n,l}, q_{n,l})$ and $(p_0, q_0) = (0, l-1)$.  
In this good region $U_{n,g}$, the function $P_n$ behaves like 
$$
|P_n(x_n,y_n)|\sim |x_n^{p_n} y_n^{q_n}| = |(y-\gamma_n(x))^{l-1}|
$$ 
with $ |y-\gamma_n(x)| \lesssim x^{m_0+\dots+ m_{n-1}}$.  
This is our second major case that will be addressed by complex interpolation in Section \ref{major2}.

%
%
%
%
%
%
%
%
%
%
%
%
%
%
%
%
%
%

 \section{The first major case:  $n=0$ and $k =1$}  \label{major1}
In this section, we need to obtain a $|\lambda|^{-\f 1 {l+1}}$ decay rate estimate of the $L_{l+1}\to L_{l+1}$ operator norm of the operator given by 
\begin{align*}
T_{0,g}f(x) = \int e^{i\lambda S(x,y)} f(y) \phi_{0,g}(x,y)dy. 
\end{align*}  
Here $\phi_{0,g}$ is essentially supported in the good region defined by the vertex $(0,l-1)$, i.e. the region where $y^{l-1}$ is the dominant term in $|P(x,y)|$. Our strategy is as follows. First, it suffices to establish the same decay estimate for the $L^{\f {l+1} l}$ operator norm of its adjoint. 
This will be done by first proving a sharp $L^2$-estimate for some truncated non-degenerate oscillatory integral operators, and second by lifting this estimate to the desired  
$L^{\f {l+1} l}$-estimate via Lemma \ref{INT001}.

We now go to the details.  
First, if $(0,l-1)$ is the only vertex in the Newton polygon, then $y^{l-1}$ is always the dominant term and $\phi_{0,g}(x,y) =\phi(x,y)$. 
This is a simple case that can be proved by quoting H\"ormander's result \cite{HOR73} directly after applying change of variables and the lifting trick below;
 see Lemma \ref{L2} and its proof below.  
In what follows, we focus on the difficult case that $\mathcal N(P)$ has vertices other than $(0,l-1)$. 
Then the good region defined by it 
has two portions, one lies in $y\gtrsim 
|x|^{m_0}$ and the other in $y \lesssim -|x|^{m_0}$. 
We can then write $\phi_{0,g}  =\phi_{0,g}^++\phi_{0,g}^-$ and $T_{0,g} =T_{0,g}^++T_{0,g}^-$.  
The treatments for them are the same and we will only handle $T_{0,g}^+$ (i.e. $y\geq |x|^{m_0}$).  
For notation simplicity, we still use $\phi_{0,g}$ and $T_{0,g}$ to denote 
$\phi_{0,g}^+$ and $T_{0,g}^+$ respectively.

 Notice that  for $\alpha = 0,1, 2$ and for all $(x,y)\in \supp \phi_{0,g}$,
\begin{align}\label{BD90}
|\partial_x^{\alpha}\phi_{0,g}(x,y)| \lesssim |y|^{-\frac \alpha {m_0}}. 
\end{align} 
Here $-\f 1 {m_0}$ is the slope of the compact edge containing $(0,l-1)$. 
It suffices to establish
\begin{align}\label{TP37}
\|T_{0,g}^*\|_{{\frac {l+1}{l}}} \lesssim |\lambda|^{-1/(l+1)},
\end{align} 
where $T_{0,g}^*$ is the adjoint of $T_{0,g}$, i.e. 
\begin{align*}
T_{0,g}^*f(y) = \int e^{-i\lambda S(x,y)} f(x) \phi_{0,g}(x,y)dx. 
\end{align*}  
Change variable $u = y^{1/ l}$ and set 
\begin{align*}
Hf(u) = \int e^{-i\lambda R(x,u)} f(x) \psi(x,u)dx
\end{align*}  
where 
\begin{align*}
\begin{cases}
R(x,u ) = S(x, u^{1/l})
\\
\psi(x,u) = \phi_{0,g}(x,u^{1/l}). 
\end{cases}
\end{align*}
We claim that 
\begin{lemma}\label{L2}
Under the above setting, the followings hold
\begin{align}\label{T35}
\|Hf\|_\infty \lesssim  \|f\|_1
\q {\rm and} 
\q\|Hf\|_2 \lesssim |\lambda|^{-1/2}\|f\|_2.
\end{align}
\end{lemma}
Then \eqref{TP37} follows from this lemma and Lemma \ref{INT001}. 
Indeed, notice that \eqref{T35} and $p_0 =2$ in Lemma \ref{INT001} imply 
\begin{align*}
(\int |Hf(u)|^p|u|^{p-2} du)^{1/p}  \lesssim  |\lambda|^{1/p-1}
\end{align*}
for $1<p\leq 2$. Setting $p =\frac {l+1} l$ yields
\begin{align*}
\left(\int \left |\int e^{i\lambda S(x,u^{1/l})} \phi(x,u^{1/l})f(x)dx\right|^{(l+1)/l}|u|^{1/l-1} du \right)^{l/(l+1)} \lesssim  |\lambda|^{-1/(l+1)}
\end{align*}
and (\ref{TP37}) follows by setting $u = y^l$.  

The first estimate in (\ref{T35}) is obvious due to the compact support of $\psi$ and what remains is the verification of the $L^2$-estimate. 
Notice that in the case $(0,l-1)$ is the only vertex, $\psi$ is just a normal cut-off (without the truncation). This $L^2$-estimate is just the classical result of H\"ormander. However, in this truncated version, if one follows the arguments from \cite{HOR73} line by line, 
there will be a $\log (2+|\lambda|)$ loss.  
As it was observed by Phong, Stein and Sturm \cite{PSS01}, 
a key to avoid this loss is to couple $TT^*$ arguments with the Hardy-Littlewood maximal functions, rather than 
Schur's test.  Details are given below.  

\begin{proof}[Proof of the $L^2$-estimate]
Let $H^*$ denote the adjoint operator of $H$, then
\begin{align*}
HH^*f(u) = \int K(u,z) f(z)dz 
\end{align*}
where
\begin{align}\label{MK}
K(u,z) = \int e^{-i\lambda(R(x,u) -R(x,z))}\psi(x,z)\psi(x,u)dx.
\end{align}
Set $\delta(\cdot) = |\cdot|^{\frac 1  {l \cdot m_0}}$ and $\delta  = \min\{\delta(u), \delta(z)\}$. One has  
$
|x| \lesssim \delta
$
given $\psi(x,z)\psi(x,u) \neq 0$. Consequently,   
\begin{align}\label{K1}
|K(u, z)| \lesssim \delta. 
\end{align}
Another estimate we need is 
\begin{align}\label{DD89}
|K(u, z)| \lesssim \delta ^{-1}  ( \lambda |u-z|)^{-2}. 
\end{align}
Together, they yield 
\begin{align}\label{KKK}
|K(u,z) |  \lesssim  \frac  \delta {1+ (\lambda \delta |u-z|)^2}.
\end{align}
We postpone the proof of \eqref{DD89} but focus on the more interesting part:  the deduction of the $L^2$-estimate from \eqref{KKK}.  
As it was mentioned above, the routine approach, namely the Schur test, is not efficient in estimating the kernel $K(u,z)$.  
A key is to exploit the Hardy-Littlewood maximal function. 
First, 
one can majorize the left hand side of \eqref{KKK} by the sum of 
\begin{align*}
K_1(u,z) =  \frac  {\delta(u) }{1+ (\lambda \delta(u) |u-z|)^2}
\q {\rm and}\q 
K_2(u,z) =  \frac  {\delta(z)} {1+ (\lambda \delta(z) |u-z|)^2}.
\end{align*}
If $M$ is the Hardy-Littlewood maximal operator, then 
\begin{align*}
\left |\int K_1(u,z) f(z) dz \right | \lesssim \lambda^{-1} Mf(u)
\q {\rm and }\q 
\left|\int K_2(u,z) g(u) du \right| \lesssim \lambda^{-1} Mg(z).
\end{align*}
As a consequence,  
\begin{align*}
\iint |K(u,z) f(z)g(u) |du dz  \lesssim \lambda^{-1} \left(\int Mf(u) |g(u)| du +\int |f(z)| Mg(z)dz\right).
\end{align*}
By the Cauchy-Schwarz inequality and the Hardy-Littlewood maximal functions, the above is
 bounded (up to multiple constant) by $\lambda ^{-1} \|f\|_2\|g\|_2$. It is obvious this estimate yields the $L^2$-estimate  
of (\ref{T35}).

What remains is the routine homework to verify \eqref{DD89}.  
Set 
\begin{align*}
\begin{cases}
A(x) = \left(\partial_x (R(x,u) -R(x,z))\right)^{-1}
\\
B(x) = \psi(x,z)\psi(x,u).
\end{cases}
\end{align*}
By applying integration by parts twice, (\ref{MK}) becomes
\begin{align}
K(u,z) = \int \left(\frac{1}{-i\lambda} \right)^2 e^{-i\lambda(R(x,u) -R(x,z))} \Big((A(x)B(x))'A(x)\Big )' dx, 
\end{align}
 that is 
\begin{align}\label{OP90}
|K(u,z)| \lesssim \lambda^{-2} \delta  \left\|\Big((AB)' A\Big)'\right\|_\infty.
\end{align}
Thus we need to find upper bounds for $\|A^{(\alpha)}\|_\infty$ and $\|B^{(\alpha)}\|_{\infty}$ for $\alpha =0, 1, 2$.  
In what follows, $\alpha$ is equal to $0$, $1$ or $2$.  
The estimates (\ref{BD90}) imply 
\begin{align*}
\|B^{(\alpha)}\|_\infty \lesssim \delta ^{-\alpha}. 
\end{align*}
By the chain rule, 
$ R''_{xu}(x,u) =   S''_{xy}(x,u^{1/l})(u^{1/l})' $. Notice also $|S''_{xy}(x,y)| \sim y^{l-1}$. Thus 
 \begin{align}
 | R''_{xu}(x,u)| \sim 1. \label{R7}
 \end{align}
 We also have 
\begin{align}\label{R8}
|\partial_x^{\alpha} R''_{xu}(x,u) | \lesssim |u|^{-\frac \alpha {l\dot m_0}}.
\end{align}
Indeed, consider the Taylor expansion of $S''_{xy}$
\begin{align*}
S''_{xy} (x,y) = \sum c_{p,q} x^py^q.
\end{align*}
Since 
$(0,l-1)$ lies in the compact edge (a supporting line) of slope $-1/m_0$, for each $(p,q)$ with $c_{p,q}\neq 0$, we have   
\begin{align}
 \frac{p}{m_0}+q  \geq (l-1)\label{VE23}
\end{align}
and equality holds if and only if $(p,q)$ is in this edge. 
Then
\begin{align*}
 R''_{xu}(x,u) = \sum{c_{p,q}'}x^p u^{q/l +(1/l-1)}  
\end{align*}
and  
\begin{align*}
\partial_{x}^{\alpha} R''_{xu}(x,u) = \sum{c_{p,q}''}x^{p-\alpha} u^{q/l +(1/l-1)}
\end{align*}
for some constants $c_{p,q}' $ and $c_{p,q}''$.  
By (\ref{VE23}), in the region $|x|\lesssim (u^{\frac 1 l})^{\frac1{m_0}}(=y^{\f1 {m_0}})$, 
\begin{align*}
|\partial_{x}^\alpha R''_{xu}(x,u) |\lesssim  \sum{c_{p,q}''}|(u^{\frac{1}{l m_0}})^{p-\alpha} u^{q/l +(1/l-1)} | \lesssim |u|^{-\frac{\alpha}{l m_0}} ,
\end{align*}
 which is (\ref{R8}).

 Employing the mean value theorem gives  
\begin{align*}
\|A^{(\alpha)}\|_\infty \lesssim \delta^{-\alpha} |u-z|^{-1}.
\end{align*}
Then (\ref{DD89}) follows from (\ref{OP90}) , the product rule and above estimates for $\|A^{(\alpha)}\|_\infty$ and $\|B^{(\alpha)}\|_\infty$.

\end{proof}

%
%
%
%
%
%
%
%
%
%
%
%
%
%
%
%
%
%

 \section{The second major case: $n\geq 1$ and $\nu = 0$} \label{major2}
We need to handle the case $\nu =0$ and $(p_n,q_n) =(p_{n,l}, q_{n,l})$, which implies  
\begin{align*}
\begin{cases}
k =  1
\\
l-1 = s_0 = \dots = s_{n-1}=q_{n,l}. 
\end{cases}
\end{align*}
In this good region $U_{n,g}$, one has 
\begin{align*}
\label{fac09}
|P(x,y)|\sim |x_n^{p_{n,l}}y_n^{q_{n,l}}| \sim |(y -\gamma(x))^{l-1}|.
\end{align*}
Here 
\begin{align*}
\gamma (x) = r_0x^{m_0} +\dots+ r_{n-1} x^{m_0+\dots+m_{n-1}}+\xi(x)
\end{align*}
where $\xi(x)$ is the sum of the remaining terms (if any) of higher degree in $x$,
and $U_{n,g}$ is contained in   
\begin{align}
|y -\gamma(x)| \lesssim  |x^{m_0+\dots+m_{n-1}}| \lesssim  |x^{m_0}| \,.
\end{align}
We only dyadically decompose $U_{n,g}$ along the $x$ variable but not the $y$ variable. 
Set 
\begin{align}\label{up90}
\phi_{n,g,\sigma}(x,y) = \sum_{\rho}\sum_{\mathcal R\in U_{n,g}(\sigma,\rho)}\phi_{\mathcal R}(x,y) 
\end{align}
and 
\begin{align}
T_{n,g,\sigma} f(x) = \int e^{i\lambda S(x,y)} f(y) \phi_{n,g,\sigma}(x,y) dy. 
\end{align}
By the same almost-orthogonality arguments from Section \ref{minor}, it suffices to handle a fixed $\sigma$.
We need to show there is a constant $C$ uniform for all small $\sigma$ such that 
\begin{align}\label{TP1}
\|T_{n,g,\sigma} f \|_{l+1} \leq C |\lambda|^{-\frac 1 {1+l}} \|f\|_{l+1}. 
\end{align}
We now let $\sigma$ be a fixed small number through out this section.   
Given the form of $P(x,y)$ in \eqref{fac09},  it is natural to make the change of variable   
 $u = \gamma (x)$ 
and set
\begin{align}
\begin{cases}
\psi_\sigma(u,y) = \phi_{n,g,\sigma}( \gamma^{-1}(u), y)
\\
R(u,y)  =\sigma^{m_0-1} S(\gamma^{-1}(u), y). 
\end{cases}
\end{align}
Define a new operator 
\begin{align}
\mathcal E_\sigma f(u) = \int e^{i\lambda R(u,y)} f(y) \psi_\sigma(u,y) dy. 
\end{align}
The following lemma implies (\ref{TP1}). 
\begin{lemma}\label{L7}
Under the above notation, we have 
\begin{align}\label{TP2}
\|\mathcal E_\sigma f \|_{l+1} \lesssim |\lambda|^{-\frac 1 {1+l}} \|f\|_{l+1}
\end{align}
with the implicit constant uniform for all small $\sigma$. 
\end{lemma}
Indeed, applying this lemma yield   
\begin{align*}
&\int \left|\int e^{i\lambda S(x,y)} f(y) \phi_\sigma(x,y) dy\right |^{l+1} dx
\\
=&  \int \left|\int e^{i\lambda S(\gamma^{-1}(u),y)} f(y) \phi_\sigma(\gamma^{-1}(u) ,y) dy\right |^{l+1} d\gamma^{-1}(u)
\\
\sim &   \sigma^{m_0-1} \int \left|\int e^{i\lambda  \sigma^{m_0 -1}  R(u,y)} f(y) \psi_\sigma(u ,y) dy\right |^{l+1} du
\\
\lesssim &    \sigma^{m_0-1}    (\lambda  \sigma^{m_0 -1} )^{-1} \|f\|^{l+1}_{l+1},
\end{align*}
as desired.  

The rest of this section is denoted to the proof of Lemma \ref{L7}, which is accomplished via complex interpolation in the following two subsections. 
The first one consists of the basic setups, leaving the technical details to the second one.

\subsection{The setups}
Embed $\mathcal E_\sigma$ into the following family of operators   
\begin{align*}
U_z f(u) = \int e^{i\lambda R(u,y)} f(y) K_z(u,y) dy
\end{align*}
where $z$ is a complex number and where 
\begin{align*}
K_z(u,y) = \psi_\sigma(u,y)  (|u-y| +|\lambda|^{-\frac {1}{l+1}})^{z}.
\end{align*}
Its adjoint is given by 
\begin{align*}
U^*_z f(y) = \int e^{-i\lambda R(u,y)} f(u) K_{\bar z}(u,y) du.
\end{align*}
Notice that when $z=0$, $U_0 =\mathcal E_\sigma$. Next we define a variant of the $H^1$ space; see \cite{PS86} and \cite{PAN91}.
\begin{definition}
Let $I$ be an internal with center $C_I$. An atom is a function $a(y)$ which is supported on $I$, so that
\begin{align*}
|a(y)| \leq \frac 1 {|I|}
\q{\rm and }\q
\int_I e^{-i\lambda R(u,C_I) }a(u) du= 0. 
\end{align*}
\end{definition}
The space $H_E^1$ consists of the subspace of $L^1$ of functions $f$ which can be written as $f = \sum_j{\mu_j} a_j$, 
where $a_j$ are atoms, and $\mu_j\in\mathbb C$, with $\sum|\mu_j| <\infty$. Consequently, we define ${\rm BMO}_E$ as
the dual space of $H_E^1$.  

Lemma \ref{L7} is then a consequence of the following lemma and complex interpolation; see \cites{PS86, FS72}. 

\begin{lemma}\label{CP01}
Under the above setting, one has
\begin{enumerate}
\item  When ${\rm Re}(z) =-1$,  $U_z$ maps from $L^\infty$ to ${\rm BMO}_E$ with operator norm $O(1+|z|^2)$ and its adjoint $U_z^*$
maps from $H^1_E$ to $L^1$ with the same norm.  
\item  When ${\rm Re}(z) = \frac {l-1} 2$,  $U_z$ maps from $L^2$ to itself with operator norm $O(|\lambda|^{-1/2}|z|^2)$. 
\end{enumerate} 
\end{lemma}

By duality, 
the two statements in the first part are equivalent and it suffices to establish the $H^1_E$ to $L^1$ boundedness of $U_z^*$. 
In the proof, we will also assume ${\rm Im}(z) =0$ and when ${\rm Im} (z)\neq 0$ there will be an extra $O(1+|{\rm Im} (z)|^2)$ factor.  
The proof will appear in the next subsection.  

The second part of this lemma is essentially the damped oscillatory integral operators studied by Phong-Stein \cites{PS94-2, PS97-2}.  
In the region $|y-u|\lesssim |\lambda|^{-\f 1 {l+1}}$, 
the desired estimate can be obtained by 
passing the absolute value into the integral of $U_zf$ 
and applying the size estimate Lemma \ref{size1}.
 In the region $|y-u|\gtrsim |\lambda|^{-\f 1 {l+1}}$, $K_{\frac{l-1}{2}}(u,y)$ is essentially the damping factor $|R''_{uy}(u,y)|^{\f 1 2}$
 since   $
|R''_{uy}(u,y)| \sim  |u-y|^{l-1}. 
$
We refer the readers to \cite{PS94-2, PS97-2} for its estimate.

\subsection{Proof of the $H^1_E\to L^1$ estimate}

Let $b$ be an atom, i.e. a function supported in some interval $[A -\delta, A+\delta]$ with
\begin{align*}
&|b(u)| \leq \frac 1 {2\delta}\quad \mbox{and} \quad
 \int e^{-i\lambda R(u,A)} b(u) du =0. 
\end{align*}
To prove the first part of Lemma \ref{CP01}, it suffices to show that there is a constant $C$ independent of $b$ such that 
\begin{align*}
\int |U^*_{-1} b (y) |dy  \leq C.
\end{align*}

Following same arguments in Section \ref{minor} it is not difficult to show $\|U^*_{-1}\|_{2} \lesssim 1$.
Coupled with the Cauchy-Schwarz inequality, this estimate yields  
\begin{align*}
\int_{|y-A| \leq 100 \delta } |U^*_{-1} b (y) |dy   \lesssim \delta^{1/2} \|b\|_2  \lesssim 1. 
\end{align*}
Hence, in what follows, we focus on $|y-A| >100 \delta \geq 100 |u-A|$, which also implies $|y-A | \sim |y -u|$ for $u \in \supp b$. 
Let $\tau >100 \delta $ be any dyadic number and let $\chi_{A,\tau} $ denote a non-negative smooth function supported in
$|y-A| \sim \tau$.
Define  
$r = (\lambda \delta )^{- 1 /l }$ 
and split the discussion into two regions as follow:  
\begin{enumerate}
\item[] \textbf{Region 1:} $ \tau\geq r$ 
\item[] \textbf{Region 2:} $100 \delta  <\tau \leq r$ 
\end{enumerate} 

Region 1 is handled via $L^2$ theory by exploiting the high-oscillation feature of the phase. 
For each such $\tau$, in the support of $\chi_{A,\tau}$ one has  
\begin{align*}
| R''_{uy}(u,y)| \sim  |u-y|^{l-1}\sim |y-A|^{l-1}\sim \tau ^{l-1}.
\end{align*}
Consequently, by applying Lemma \ref{os2},  the operator given by 
\begin{align*}
\chi_{A,\tau}( y) U^*_{-1} b(y) =\chi_{A,\tau}( y)  \int e^{-i\lambda R(u,y)} b(u) \psi_\sigma(u,y)  (|y-u| +|\lambda|^{-1/(l+1)})^{-1} du
\end{align*} 
satisfying 
\begin{align*}
\left(\int |\chi_{A,\tau}( y) U^*_{-1} b(y)|^2 dy\right)^{1/2} \leq (\lambda |\tau|^{l-1})^{-1/2} \tau^{-1} \|b\|_2 \sim (\lambda \delta \tau^{l+1})^{-1/2}.
\end{align*}
Thus, Cauchy-Schwarz' inequality implies 
\begin{align*}
\sum_{\tau \geq r }\int | \chi_{A,\tau}( y) U^*_{-1} b(y)  |dy \lesssim   \sum_{\tau \geq r}  \tau^{1/2}  \| \chi_{A,\tau}(\cdot)U^*_{-1} b(\cdot)\|_2 \lesssim  1.  
\end{align*}

To handle Region 2, we need to exploit the ``mean zero" property of $b$ as well   
as the low-oscillation feature of the phase.  
We divide $U_{-1}b(y) $ into the sum of 
\begin{align*}
I_1 = \int (e^{-i\lambda R(u,y)} - e^{-i\lambda (R(u,A)+R(A, y)  -R(A,A))  } )b(u) K_{-1}(u,y) du
\end{align*}
and 
\begin{align*}
I_2 = \int  e^{-i\lambda (R(u,A)+R(A, y)  -R(A,A))  } b(u) K_{-1}(u,y) du.
\end{align*}
To address the first part, we apply the mean value theorem and obtain 
\begin{align*}
 | R(u,y) - (R(u,A)+R(A, y)  -R(A,A)) |
=|u-A| |y-A| \partial_u\partial_y R(u_0,y_0) | \,,
\end{align*}
for some $u_0$ between $u$ and $A$ and some $y_0$ between $A$ and $y$.  
This is bounded above by  $ \delta |y-A|^l$ since 
$|y_0 -u_0| \lesssim |y-A| $
and 
\begin{align*}
|R(u_0,y_0) | \sim |y_0 -u_0|^{l-1} \lesssim |y-A|^{l-1}.
\end{align*}
 Passing the absolute value into the integral, applying 
Taylor's expansion to the phase and 
utilizing the facts $|K_{-1}(u,y)| \lesssim |y-A|^{-1}$ and $  \|b\|_1 \lesssim 1$,
we obtain 
\begin{align*}
|I_1| \lesssim   \int \lambda \delta |y-A|^l |b(u)|  |K_{-1}(u,y) | du \leq  \lambda \delta |y-A|^{l-1}.
\end{align*}
Consequently,  
\begin{align*}
\int_{|y-A|<r} |I_1| dy  \lesssim 1. 
\end{align*}
To control $I_2$, notice that 
\begin{align*}
|I_2| = \left|\int  e^{-i\lambda R(u,A) } b(u) K_{-1}(u,y) du\right| 
\end{align*}
which is majorized by the sum of  
\begin{align*}
I_3  =\left |\int  e^{-i\lambda R(u,A) } b(u) (K_{-1}(u,y) -K_{-1}(A,y)) du\right| 
\end{align*}
and 
\begin{align*}
I_4  =\left |K_{-1}(A,y)\int  e^{-i\lambda R(u,A) } b(u)  du\right|.
\end{align*}
Notice that $I_4 = 0$, since $b$ is an atom. 
We also claim that $I_3$ is bounded by $\delta |y-A|^{-2}$, which gives the desired estimate since 
\begin{align*}
\int_{|y-A| \geq 100 \delta } I_3 dy  \lesssim \int_{|y-A| \geq 100 \delta }\delta |y-A|^{-2} dy  \lesssim 1. 
\end{align*}
It remains to verify this bound for $I_3$.  
Indeed, 
by the mean value theorem again,  
\begin{align*}
I_3 \lesssim  |u-A|\,\, |\partial_ u K_{-1} (u_0, y)| \,\, \|b\|_1 \sim \delta \, |\partial_ u K_{-1} (u_0, y)|
\end{align*}
for some $u_0$ between $u$ and $A$. The product rule implies that $|\partial_ u K_{-1} (u_0, y)|$ 
is controlled by 
\begin{align*}
 (|y-u_0| +\lambda^{-1/(l+1)})^{-2} |\psi_\sigma( u_0,y)|+(|y-u_0| +\lambda^{-1/(l+1)})^{-1} \cdot 
|\partial_u\psi_\sigma( u_0,y)|\, .
\end{align*}
The fact $|y-u_0| \sim |y-A|$ implies $(|y-u_0| +\lambda^{-1/(l+1)})^{-1} \lesssim |y-A|^{-1}$.
The first term is majorized by 
$|y-A|^{-2}$ since $|\psi_\sigma| \lesssim 1$.
To prove the same estimate for the second term, 
we only need to verify $|\partial_u\psi_\sigma( u_0,y)|  \lesssim   |y-A| ^{-1}$.   

First, notice that if $\psi_\sigma( u_0,y) \neq 0$, then there is a $x_0$ such that $u_0 =\gamma(x_0)$ and 
$(x_0,y) \in \supp \phi_{n,g,\sigma}$. In particular,  
$\phi_\mathcal R(x_0,y) \neq 0 $ for some $\mathcal R$. Note that there are finitely many such $\mathcal R$, whose sizes are essentially 
the same, 
 denoted by $\Delta x \times \Delta y$.   
In particular, $|y-A|\sim |y-u_0|  = |y -\gamma(x_0)| \lesssim \Delta y$
and 
$$
|\partial_x \phi_{n,g,\sigma}(x_0,y)| \lesssim |\p_x \phi_\mathcal R(x_0,y)| \lesssim \Delta x^{-1}.
$$
Since $x_0\sim \sigma$ and 
$|(\gamma^{-1})'(u_0)| = |{\gamma'(x_0)}|^{-1}\sim  \sigma^{1-m_0}
$, the chain rule implies 
\begin{align*}
|\partial_u\psi_\sigma(u_0,y) | =  |\partial_x\phi_{n,g,\sigma} (x_0,y) | |(\gamma^{-1})' (u_0)| \lesssim \Delta x ^{-1} \sigma^{1-m_0} =\Delta y ^{-1}
\lesssim |y-A|^{-1},
\end{align*}
as desired.


\end{document}